\documentclass[12pt]{amsart}

\newtheorem{theorem}{Theorem}[section]
\newtheorem{lemma}[theorem]{Lemma}

\newtheorem{proposition}[theorem]{Proposition}

\usepackage[bookmarks]{hyperref}

\usepackage{graphicx,amscd,amsthm,amsfonts,amsopn,amssymb,verbatim,enumerate,xcolor}

\usepackage[letterpaper,left=1in,top=1in,right=1in,bottom=1in]{geometry}

\def\half{ \frac{1}{2}}
\def\D{\partial}

\def\nint{\mathop{\diagup\kern-13.0pt\int}}

\def\bas{\begin{align*}}
\def\eas{\end{align*}}
\def\bi{\begin{itemize}}
\def\ei{\end{itemize}}

\def\emph#1{{\it #1}}

\theoremstyle{definition}

\numberwithin{equation}{section}

\begin{document}

\title{Well-posedness for the dispersive Hunter-Saxton equation}

\author{Albert Ai}
\address{Department of Mathematics, University of Wisconsin, Madison}
\email{aai@math.wisc.edu}
\author{Ovidiu-Neculai Avadanei}
\address{Department of Mathematics, University of California at Berkeley}
\email{ovidiu\_avadanei@berkeley.edu}

\begin{abstract}
This article represents a first step towards understanding the well-posedness for the dispersive Hunter-Saxton equation. This problem arises in the study of nematic liquid crystals, and although the equation has formal similarities with the KdV equation, the lack of $L^2$ control gives it a quasilinear character, with only continuous dependence on initial data.

Here, we prove the local and global well-posedness of the Cauchy problem using a normal form approach to construct modified energies, and frequency envelopes in order to prove the continuous dependence with respect to the initial data.
\end{abstract}

\maketitle

\addtocontents{toc}{\protect\setcounter{tocdepth}{1}}
\tableofcontents

\section{Introduction}

In this article we consider the Cauchy problem for the dispersive Hunter-Saxton equation
\begin{equation}\label{HS}
\left\{\begin{aligned}
&u_t+uu_x+u_{xxx}= \frac12 \partial_x^{-1}(u_x^2)\\
&u(0)=u_0,
\end{aligned}\right.
\end{equation}
where $u$ is a real-valued function $u:[0,\infty)\times\mathbb{R}\rightarrow\mathbb{R}$. Due to the Galilean invariance of \eqref{HS}, we may fix a definition for $\partial_x^{-1}$,
\[
\partial_x^{-1}f(x)=\int_{-\infty}^xf(y)\, dy,
\]
where $f\in L_x^1(\mathbb{R})$.
The dispersive Hunter-Saxton equation is a perturbation of the Hunter-Saxton equation
\begin{equation}\label{ndHS}
\begin{aligned}
&u_t+uu_x=\frac12 \partial_x^{-1}(u_x^2),
\end{aligned}
\end{equation} 
which was introduced in \cite{16}
as an asymptotic model for the formation of nematic liquid crystals under a director field. The Hunter-Saxton equation \eqref{ndHS} is completely integrable \cite{17, 20} with a bi-Hamiltonian structure \cite{19}. In the periodic case, the local well-posedness and blow up phenomena were studied in \cite{16,24}, while global weak solutions were studied in \cite{4,5}. For the non-periodic case, the Cauchy problem local well-posedness and blow up were studied in \cite{YY}.

The Hunter-Saxton equation is also the high frequency limit of the Camassa-Holm equation,
\begin{align*}
(1-\partial_x^2)u_t=3uu_x-2u_xu_{xx}-uu_{xxx}.
\end{align*}
The local well-posedness and ill-posedness of the Camassa-Holm equation were studied in \cite{9,21,22}. The global existence of strong solutions and blow up phenomena were investigated in \cite{6,7,8,9}.

The dispersive Hunter-Saxton equation \eqref{HS} first appeared in \cite{15} as a dispersive regularization of \eqref{ndHS}. Complete integrability was later observed in \cite{CI}.

In this paper, we initiate the study of the well-posedness for the dispersive Hunter-Saxton equation \eqref{HS}. For this purpose, we use the conserved quantities
\begin{align*}
    E_1(t)&=\int_{\mathbb{R}}u_x(t)^2 \, dx,\\
    E_2(t)&=\int_{\mathbb{R}}u_{xx}(t)^2-u(t)u_x(t)^2 \, dx.
\end{align*}
Throughout, we denote 
\[
X^s=L_x^\infty\cap\dot H_x^1 \cap \dot H_x^{1+s},
\]
where $s\in[0, 1]$. For brevity, we denote $X=X^1$. Our first main result is the following local well-posedness statement:

\begin{theorem}\label{Local well-posedness}
The dispersive Hunter-Saxton equation \eqref{HS} is locally well-posed in $X$. Precisely, for every $R>0$, there exists $T=T(R)>0$ such that for every $u_0\in X$ with $\|u_0\|_X<R$, the Cauchy problem \eqref{HS} has a unique solution $u \in C([0, T], X)$. Moreover,  the solution map $u_0\mapsto u$ from $X$ to $C([0, T], X)$ is continuous.
   \end{theorem}

   In both the dispersive and nondispersive cases of the Hunter-Saxton equation, the main difficulty is that the source term $\frac12 \partial_x^{-1}(u_x^2)$ is unbounded in any $L^p$ space if $p < \infty$, and in particular, in $L^2$. As a result, it is necessary to consider the problem assuming only pointwise $L^\infty$ control on $u$, similar to the analysis in \cite{YY} for the nondispersive case \eqref{ndHS}. Further, the lack of spatial decay obstructs direct access to local smoothing estimates, so that \eqref{HS} exhibits quasilinear behavior even in the presence of KdV-like dispersion. In particular, our solutions exhibit only continuous dependence on the initial data, instead of Lipschitz dependence.
   
   Our proof follows a bounded iterative scheme which treats separately the high and low frequency components. To prove continuous dependence on the initial data in our quasilinear setting, we have used frequency envelopes, introduced by Tao in \cite{25}. A systematic presentation of the use of frequency envelopes in the study of local well-posedness theory for quasilinear problems can be found in the expository paper \cite{LWN}, which we broadly follow in the present work.
   
Our second result is the following global well-posedness statement:
\begin{theorem}\label{Global well-posedness}
The Cauchy problem \eqref{HS} is globally well-posed in $X$. Moreover, for every $t\geq 0$, we have the global in time bounds
\begin{align*}
    \|u(t)\|_{L_x^\infty}&\lesssim \|u_0\|_{X^0}+t(E_1+E_1^{1/2}),
    \\
    \|u(t)\|^2_{\dot H_x^2}&\lesssim \|u_0\|_{\dot H_x^2}^2+\|u_0\|_{X^0}E_1 + t(E_1+E_1^{1/2})E_1.
\end{align*}

\end{theorem}
Its proof relies on Theorem \ref{Local well-posedness} and on the conserved quantities $E_1(t)$ and $E_2(t)$. We remark that the $L^\infty$ estimate holds even for solutions which are only in $X^0 = L^\infty \cap \dot H^1$.

\
   
   Using the $X^1$ well-posedness as a starting point, our third and fourth results extend well-posedness to lower regularity data:
   
\begin{theorem}\label{Low regularity local well-posedness}
For each $ s\in(\frac{1}{2},1)$, the Cauchy problem \eqref{HS} is locally well-posed in $X^s$. 
\end{theorem}
   
The local well-posedness of Theorem~\ref{Low regularity local well-posedness} is in the same sense as in Theorem \ref{Local well-posedness}. Here, we leverage Theorem \ref{Local well-posedness} to construct $X^s$ solutions as limits of sequences of smooth solutions, by proving an estimate for differences of two solutions in order to establish convergence. This in turn is a consequence of an estimate for the linearized equation associated to \eqref{HS},
  \begin{equation}\label{linearized HS}
      w_t+(uw)_x+w_{xxx}=\partial_x^{-1}(u_xw_x).
  \end{equation}
    
 \begin{theorem}\label{Low regularity global well-posedness}
For each $ s\in(\frac{1}{2},1)$, the Cauchy problem \eqref{HS} is globally well-posed in $X^s$. Moreover, for every $t\geq 0$,
\begin{equation}\label{low-reg-growth-bounds}
\begin{aligned}
    \|u(t)\|_{L_x^\infty}&\lesssim \|u_0\|_{X^0}+t(E_1+E_1^{1/2})\\
    \|u(t)\|_{\dot H_x^{1+s}}^2 &\lesssim \langle t\rangle^4 E_1^2 \langle \|u_0\|_{X^0}+ E_1\rangle^2 +\|u_0\|^2_{\dot H_x^{1+s}}.
\end{aligned}
\end{equation}
\end{theorem}

   \

  To prove Theorem \ref{Low regularity global well-posedness}, we construct a modified energy functional for $\dot H^{1+s}$ which is based on the quadratic normal form for \eqref{HS}. The approach of constructing normal form inspired modified energies in the quasilinear setting was first introduced by Hunter-Ifrim-Tataru-Wong \cite{Burgers-Hilbert} which considered the Burgers-Hilbert equation. This approach was further developed in the gravity water wave setting by Hunter-Ifrim-Tataru in \cite{HIT}, which established almost-global well-posedness, and in the Benjamin-Ono setting by Ifrim-Tataru \cite{BO} which established dispersive decay.
  
  \
  
Our paper is organized as follows. In Section~\ref{s:prelim}, we present some existence results at various degrees of regularity for linear equations that arise throughout the proofs of the main results. In Section~\ref{s:lwp}, using an iterative scheme, we prove the higher regularity local well-posedness result, while in Section~\ref{s:gwp}, by using the conserved quantities $E_1$ and $E_2$, we show that the dispersive Hunter-Saxton equation \eqref{HS} is globally well-posed.

Section~\ref{s:nf} analyzes a modified energy for the equation, which is based on the normal form associated to the Hunter-Saxton equation, in order to obtain bounds on the growth of the $X^s$-norm, whereas Section~\ref{s:lin} discusses an estimate for the linearized equation \eqref{linearized HS}, as well as one for differences of solutions. These results are then used to prove the low regularity local well-posedness result in Section~\ref{s:lowreg}.

\subsection{Acknowledgements}

This material is based upon work supported by the National Science Foundation under Grant No. DMS-1928930 while the authors participated in a program hosted by the Mathematical Sciences Research Institute in Berkeley, California, during the Spring 2021 semester.

The authors would like to thank Mihaela Ifrim and Daniel Tataru for many helpful discussions. We also thank John Hunter for a discussion about the history and derivation of the Hunter-Saxton equation.

   \section{Preliminaries}\label{s:prelim}
   
   In this section we state and prove some results that will be used in the sequel. We begin by studying well-posedness for a linear equation which will be used in the iteration for the proof of Theorem \ref{Local well-posedness}.
   
   We first prove well-posedness and energy estimates for initial data in $L^2$.
   
   \begin{lemma}\label{L_x^2 well-posedness}
   Let $T>0$, $a,b\in L_t^\infty([0, T], \dot{W}^{1,\infty})$, $F\in L_t^1([0, T], L_x^2)$, $v_0\in L_x^2(\mathbb{R})$. Then the Cauchy problem
   \begin{equation}\label{linear}\left\{\begin{aligned}
        &v_t+av_x+b_xv+v_{xxx}=F\\
        &v(0)=v_0
        \end{aligned}\right.
    \end{equation}\\
    admits a unique solution $v\in L_t^\infty([0, T], L_x^2)$ which satisfies the energy estimate
    \begin{equation*}
     \frac{d}{dt} \|v\|_{L^2}^2 \lesssim    \|F\|_{L_x^2}\|v\|_{L_x^2}+(\|a_x\|_{L_x^\infty}+\|b_x\|_{L_x^\infty})\|v\|_{L_x^2}^2.
    \end{equation*}
   \end{lemma}
   \begin{proof}
    Let us assume that $v$ is a solution to the Cauchy problem. We have
    \begin{equation}\label{timeD}
    \begin{aligned}
        \frac{d}{dt}\int_{\mathbb{R}}v^2(t) \, dx&=2\int_{\mathbb{R}}v(t)v_{t}(t) \, dx\\
        &=2\int_{\mathbb{R}}v(t)(F(t)-a(t)v_x(t)-b_x(t)v(t)-v_{xxx}(t)) \, dx\\
        &=2\int_{\mathbb{R}}v(t)F(t)\,  dx+\int_{\mathbb{R}}a_x(t)v^2(t) \, dx - 2\int_{\mathbb{R}}b_x(t)v^2(t) \, dx\\
        &\lesssim \|v(t)\|_{L_x^2}\|F(t)\|_{L_x^2}+\|v(t)\|_{L_x^2}^2(\|a_x(t)\|_{L_x^\infty}+\|b_x(t)\|_{L_x^\infty}).
    \end{aligned}
    \end{equation}
    We obtain the desired energy estimate, which also establishes uniqueness. 
    
    \
    
    It remains to show existence, for which we follow a standard duality argument.  We first determine the adjoint problem. For an arbitrary $w$, a formal computation shows that
    \begin{align*}
      \int_0^T\int_{\mathbb{R}}(v_t+av_x+b_xv+v_{xxx})w\, dx\, dt &=\int_{\mathbb{R}}v(T)w(T)\, dx-\int_{\mathbb{R}}v(0)w(0)\, dx\\
        &\quad -\int_0^T\int_{\mathbb{R}}(w_t+aw_x+a_xw+b_xw+w_{xxx})v\, dx\, dt.
    \end{align*}
    We write $\displaystyle w_t+aw_x+(a_x+b_x)w+w_{xxx}=G$ and $w(T) = w_T$.
    Thus,
    \begin{align*}
       \int_0^T\int_{\mathbb{R}}Fw\, dx\, dt+\int_{\mathbb{R}}v_0w(0)\,  dx=\int_{\mathbb{R}}v(T)w_T\,  dx-\int_0^T\int_{\mathbb{R}}Gv\, dx \, dt
    \end{align*}
    and we have the adjoint problem
    \begin{equation}\left\{\begin{aligned}
        &w_t+aw_x+(a_x+b_x)w+w_{xxx}=G\\
        &w(T)=w_T.
    \end{aligned}\right.     \end{equation}
    
    Using the energy estimate of the original equation, we have
    \[ \displaystyle \|w(t)\|_{L_x^2}\lesssim \|w_T\|_{L_x^2}+\|G\|_{L_t^1L_x^2}. \]
    In particular, we conclude that if the adjoint problem has a solution, then it is unique.
    
    \
    
    Let 
    \begin{equation*}
    \begin{aligned}
     Y &=\{(g,\tilde{G}) \in L_x^2 \times L_t^1 L_x^2([0,T]\times\mathbb{R}) \ \vert \\
     &\qquad \text{there exists }h\in L_t^\infty L_x^2\text{ solving the adjoint problem with } (w_T, G) = (g, \tilde G)\ \}.
    \end{aligned}
    \end{equation*}
    We define the functional $\alpha: Y\rightarrow\mathbb{R}$ by 
    \[\displaystyle \alpha(g,\tilde G)=\int_0^T\int_{\mathbb{R}}Fh\, dx\, dt+\int_{\mathbb{R}}v_0h(0) \, dx,\]
    which is well-defined by uniqueness for the adjoint problem. It is also bounded, as
    \begin{align*}
        \vert\alpha(g,\tilde{G}) \vert&\lesssim \|v_0\|_{L_x^2}\|h(0)\|_{L_x^2}+\|F\|_{L_t^1L_x^2}\|h\|_{L_t^\infty L_x^2}\\
        &\lesssim \|v_0\|_{L_x^2}(\|g\|_{L_x^2}+\|\tilde{G}\|_{L_t^1L_x^2})+\|F\|_{L_t^1L_x^2}(\|g\|_{L_x^2}+\|\tilde{G}\|_{L_t^1L_x^2})\\
        &\lesssim (\|v_0\|_{L_x^2}+\|F\|_{L_t^1L_x^2})(\|g\|_{L_x^2}+\|\tilde{G}\|_{L_t^1L_x^2}).
    \end{align*}
    Using the Hahn-Banach Theorem, we extend $\alpha$ to a functional $\beta$ defined on $L_x^2\times L_t^1 L_x^2$. This uniquely corresponds an element of $L_x^2\times L_t^\infty L_x^2$, whose second component is the desired solution $v$.

   \end{proof}

   We extend the previous result to the case when the initial data is in $H^1$:
   
   \begin{lemma}\label{H_x^1 well-posedness}
   Let $T>0$, $a,b\in L_t^\infty \dot{W}_x^{1,\infty}$, $b\in L_t^\infty\dot{H}_x^2$, $F\in L_t^1H_x^1$, and $v_0\in H_x^1$. Then the Cauchy problem \eqref{linear} has a unique solution $v\in L_t^\infty H_x^1$ which satisfies the energy estimate
   \begin{equation*}
       \frac{d}{dt}\|v\|_{\dot H^1_x}^2 \lesssim (\|F\|_{\dot H_x^1} + \|b_{xx}\|_{L^2_x} \|v\|_{L^\infty_x}) \|v\|_{\dot H^1_x} + (\|a_x\|_{L_x^\infty}+\|b_x\|_{L_x^\infty})\|v\|_{\dot H^1_x}^2.
   \end{equation*}
     In particular, if $u$ is a solution of the dispersive Hunter-Saxton equation \eqref{HS}, then     
     \begin{equation*}
     \frac{d}{dt} \|u_x\|_{\dot H^1_x}^2 \lesssim    \|u_x\|_{L_x^\infty}\|u_x\|_{\dot H_x^1}^2.
    \end{equation*}
    \end{lemma}
   \begin{proof}
   We first consider the regularized equation 
   \[ v_t+v_{xxx}+av_x+(b_{\leq m})_xv=F.
   \]
   By applying Lemma $\ref{L_x^2 well-posedness}$, we obtain a unique solution $v^m\in L_t^\infty L_x^2$. We first observe that $v^m\in L_t^\infty H_x^1$. Indeed, note that $v^m_x$ formally satisfies
  \begin{equation}\label{mollified}\begin{aligned}
       \tilde{v}_t+\tilde{v}_{xxx}+(a_x+(b_{\leq m})_x)\tilde{v}+a\tilde{v}_x&=F_x-(b_{xx})_{\leq m}v^m
    \end{aligned}     
    \end{equation}
    where \begin{align*}
     \|F_x-(b_{xx})_{\leq m}v^m\|_{L_t^\infty L_x^2}&\leq \|F_x\|_{L_t^\infty L_x^2}+\|(b_{xx})_{\leq m}v^m\|_{L_t^\infty L_x^2}\\
     &\leq \|F_x\|_{L_t^\infty L_x^2}+\|(b_{xx})_{\leq m}\|_{L_{t, x}^\infty}\|v^m\|_{L_t^\infty L_x^2}<\infty.
     \end{align*}
    By applying Lemma $\ref{L_x^2 well-posedness}$ once again, we obtain that \eqref{mollified} admits a unique solution $\tilde{v^m}\in L_t^\infty L_x^2$ so that $v^m_x = \tilde{v^m}$ and $v^m\in L_t^\infty H_x^1$.
    
    Using \eqref{timeD}, we find
    \[
    \frac{d}{dt}\int_{\mathbb{R}}(v^m)^2 \, dx \lesssim \|v^m\|_{L_x^2}\|F\|_{L_x^2}+(\|a_x\|_{L_x^\infty}+2\|b_x\|_{L_x^\infty})\|v^m\|_{L_x^2}^2
    \]
    and 
    \begin{align*}
        \frac{d}{dt}\int_{\mathbb{R}}(v^m_x)^2 \, dx &\lesssim \|v^m_x\|_{L_x^2}\|F_x-(b_{xx})_{\leq m}v^m\|_{L_x^2}+\|a_x+2(b_{\leq m})_x\|_{L_x^\infty}\|v^m_x\|_{L_x^2}^2\\
        &\lesssim \|v^m_x\|_{L_x^2}\|F_x\|_{L_x^2}+\|b_{xx}\|_{L_x^2}\|v^m\|_{L_x^\infty}\|v^m_x\|_{L_x^2}+(\|a_x\|_{L_x^\infty}+\|b_x\|_{L_x^\infty})\|v^m_x\|_{L_x^2}^2.
    \end{align*}
    Denoting
    \[
    E^m(t)=\int_{\mathbb{R}}(v^m(t))^2\, dx+\int_{\mathbb{R}}(v^m_x(t))^2\, dx,
    \]
    we have
    \begin{align*}
       \frac{d}{dt}E^m(t)&\lesssim (E^m(t))^{1/2}\|F(t)\|_{H_x^1}+(\|a_x(t)\|_{L_x^\infty}+\|b_x(t)\|_{L_x^\infty}+\|b_{xx}(t)\|_{L_x^2})E^m(t).
    \end{align*}
    From Gr\"onwall's lemma, we infer that
    \begin{align*}
        E^m(t)&\leq e^{\frac{C}{2}\int_0^T\|a_x(s)\|_{L_x^\infty}+\|b_x(s)\|_{L_x^\infty\cap\dot{H}_x^1}\, ds} \cdot \\
        &\quad \left(\|v_0\|_{H_x^1}+\int_0^Te^{-\frac{C}{2}\int_0^s\|a_x(\tau)\|_{L_x^\infty}+\|b_x(\tau)\|_{L_x^\infty\cap\dot{H}_x^1}\, d\tau}\|F(s)\|_{H_x^1}\, ds\right),
    \end{align*}
   uniformly in $m$ and $t\in[0,T]$.
    
    \
    
 Let $l\geq 0$ and $z=v^{m+l} - v^m\in L_t^\infty L_x^2$. We see that $z$ solves
 \begin{equation*}\begin{aligned}
     z_t+z_{xxx}+az_x+(b_x)_{\leq m+l}z&=-(b_x)_{m<\cdot \leq m+l}v^m =: H.
 \end{aligned}
 \end{equation*}
 Let $\displaystyle e := \sup_{m\geq 1}\sup_{t\in [0,T]}E^{m}(t)<\infty$. We estimate the source term:
 \begin{align*}
     \|H\|_{L_t^\infty L_x^2}&\lesssim \|(b_x)_{m<\cdot \leq m+l}\|_{L_t^\infty L_x^2}\|v^m\|_{L_{t, x}^\infty} \lesssim 2^{-m}\|(b_{m<\cdot \leq m+l})_{xx}\|_{L_t^\infty L_x^2}\|v^m\|_{L_t^\infty H_x^1} \\
     &\lesssim 2^{-m}\|b_{xx}\|_{L_t^\infty L_x^2}e^{1/2}.
 \end{align*}
 By applying the energy estimate provided by Lemma $\ref{L_x^2 well-posedness}$ with Gr\"onwall, we obtain 
 \begin{align*}
     \|z(t)\|_{L_x^2}&\leq e^{\frac{C}{2}\int_0^t \|a_x(s)\|_{L_x^\infty}+2\|b_x(s)\|_{L_x^\infty} \, ds}\left(\frac{C}{2}\int_0^te^{-\frac{C}{2}\int_0^s \|a_x(\tau)\|_{L_x^\infty}+2\|b_x(\tau)\|_{L_x^\infty}\, d\tau}\|H(s)\|_{L_x^2}\, ds \right)\\
     &\lesssim T2^{-m}e^{1/2}\|b_{xx}\|_{L_t^\infty L_x^2}.
 \end{align*}
 
 Thus, $\displaystyle v^m$ is a Cauchy sequence in $L_t^\infty L_x^2$, which means that it converges to a solution $v$. As $\displaystyle v^{m}$ is bounded in $L_t^\infty H_x^1$, Lemma $\ref{Regularity upgrade for limits}$ implies $v\in L_t^\infty H_x^1 $. The energy estimates of Lemma $\ref{L_x^2 well-posedness}$ also prove uniqueness.
 A similar computation to the one carried out for $v^m$ provides the desired energy estimate. In particular, if $u$ is a solution of \eqref{HS}, then $u_x$ is a solution of \eqref{linear} with $a = u$, $b = -u_x/2$, and $F = 0$, so that the desired estimate follows.
   \end{proof}
   
   Using this, we establish persistence of regularity for \eqref{HS}:
   
   \begin{lemma}\label{Persistence of regularity}
   Let $T>0$, and $u\in C([0, T], X)$ a solution for the dispersive Hunter-Saxton equation  \eqref{HS}. If $u(0)\in X\cap \dot{H}_x^{n + 1}(\mathbb{R})$, then $\displaystyle u\in L_t^\infty([0, T], X\cap \dot{H}_x^{n + 1})$. Furthermore, in the case $n = 2$, we have the energy estimate    
   \begin{equation*}
     \frac{d}{dt} \|u_{xx}(t)\|_{H^1_x}^2 \lesssim    \|u_x(t)\|_{L_x^\infty}\|u_{xx}(t)\|_{H_x^1}^2.
    \end{equation*}
   \end{lemma}
   \begin{proof}
   Observe that $u_{xx}$ formally satisfies
   \begin{equation*}\begin{aligned}
      v_t+uv_{x}+2u_xv+v_{xxx}=0.
   \end{aligned}
   \end{equation*}
   As $u\in L_t^\infty X$, by applying Lemma $\ref{H_x^1 well-posedness}$, we infer that the problem admits a unique solution $v\in L_t^\infty H_x^1$. In particular, $v$ solves the problem in the sense of distributions, so that $v=u_{xx}$ and $\displaystyle u\in L_t^\infty (X\cap \dot{H}_x^3)$, along with the energy estimate, as desired.
   
   For $n > 2$, observe that $\D_x^n u$ formally satisfies
      \begin{equation}\begin{aligned}
      v_t+uv_{x}+2u_xv+v_{xxx}=P(u_{xx}, ..., \D_x^{n - 2} u),
   \end{aligned}
   \end{equation}
   where $P$ is a quadratic polynomial. The result follows by induction and Lemma $\ref{H_x^1 well-posedness}$.
   \end{proof}

   We now establish the following $L^\infty$ estimate that will be used in the proof of several other results, including the iteration for the proof of Theorem \ref{Local well-posedness}:
   
   \begin{lemma}\label{inftyest}
   Let $T>0$, $a\in L_t^\infty([0, T], W^{1,\infty})$, and $w \in L_t^\infty([0, T], L^\infty)$ satisfy
   \begin{equation}\begin{aligned}
        w_t+aw_x+w_{xxx}&= f.
    \end{aligned}
    \end{equation}
    Then $w$ satisfies
    \[
    \frac{d}{dt} \|w_{\leq 0}\|_{L^\infty} \lesssim 
    \| w_{\leq 0}\|_{L^\infty} + \| f_{\leq 0}\|_{L^\infty} + \|a\|_{W^{1,\infty}} \| w\|_{L^\infty} 
    \].
    
   \end{lemma}
   \begin{proof}
   By applying the frequency projection $P_{\leq 0}$, we obtain
\begin{align*}
    (w_{\leq 0})_t+(aw_x)_{\leq 0}+(w_{\leq 0})_{xxx}=f_{\leq 0}
    \end{align*}
    and estimate 
    \begin{equation*}
        \begin{aligned}
        \|(aw_x + w_{xxx})_{\leq 0}\|_{L^\infty} &\lesssim \| ((aw)_x - (a_xw))_{\leq 0}\|_{L^\infty}  + \|(w_{\leq 0})_{xxx}\|_{L^\infty} \\
           &\lesssim (\|a\|_{L^\infty} + \|a_x\|_{L^\infty})\|w\|_{L^\infty} + \|w_{\leq 0}\|_{L^\infty}.
        \end{aligned}
    \end{equation*}
   \end{proof}
   
   Lastly, we observe a technical result which will be used in the proof of Theorem \ref{Local well-posedness} to show that the solution of \eqref{HS} has the desired regularity:
   
   \begin{lemma}\label{Regularity upgrade for limits}
   Let $T>0$ and $\{v^n\}_{n\geq 0}\in L_t^\infty([0, T], H_x^1)$ be a bounded sequence such that
   \[
   v^n \rightarrow v \in L_t^\infty([0, T], L_x^2).
   \]
   Then $v\in L_t^\infty([0, T], H_x^1)$.
   \end{lemma}
   \begin{proof}
   Let $M>0$ be such that $\displaystyle \|v^n\|_{L_t^\infty H_x^1}\leq M$ for every $n\geq 0$. Fix $t\in [0,T]$ such that $v^n(t)$ converges to $v(t)$ in $L_x^2(\mathbb{R})$, and $\displaystyle \|v^n(t)\|_{H_x^1}\leq M$. It suffices to show that $\|v(t)\|_{H_x^1}\leq M$, independently of $t$. We omit $t$ in the notations below. 
   
   As $v^n$ is bounded in $H_x^1(\mathbb{R})$, which is a Hilbert space and hence reflexive, we infer that there exists a subsequence $\{v^{n_k}\}_{k\geq 0}$ that converges weakly to some $g\in H_x^1(\mathbb{R})$. In particular, $v^{n_k}$ converges to $g$ in the sense of distributions. On the other hand, $v^n$ converges to $v$ in $L_x^2(\mathbb{R})$ and in the sense of distributions, so $v=g \in H_x^1(\mathbb{R})$.
   
   Let $w\in H_x^1(\mathbb{R})$ with $\|w\|_{H_x^1}=1$, and observe that
   \begin{align*}
       \vert\langle v,w\rangle\vert=\lim_{\substack{k\rightarrow\infty}}\vert\langle v^{n_k},w\rangle\vert\leq \lim_{k \rightarrow \infty} \|v^{n_k}\|_{H_x^1}\leq M.
   \end{align*}
   We infer that $\|v\|_{H_x^1}\leq M$. This finishes the proof.
   
   \end{proof}
   
   \section{Local well-posedness}\label{s:lwp}
   
   In this section we prove Theorem~\ref{Local well-posedness}.
   
   \
   
   Let $C > 0$ be a large absolute constant which may vary from line to line, and let small $T > 0$ be fixed later. Let $\|u_0\|_X < R$. We inductively define a sequence $\{u^n\}_{n \geq 0}\in L_{t,x}^\infty([0,T]\times\mathbb{R})$. For $n=0$  we set  $u^0(t,x)=u_0(x)$. For $n > 0$, we will set $u^{n+1}\in L_{t,x}^\infty([0,T]\times\mathbb{R})$ as the unique solution of the Cauchy problem
   \begin{equation}\label{iterate}
   \left\{\begin{aligned}
       &u^{n+1}_t+u^{n+1}_{xxx}+u^nu_x^{n+1}=\frac{\partial_x^{-1}((u_x^n)^2)}{2}, \\
       &u^{n+1}(0)=u_0.
       \end{aligned}\right.
   \end{equation}

    \subsection{Existence and uniform bounds for \eqref{iterate}} Here we show existence and estimates for \eqref{iterate} in $L^\infty_t([0, T], X)$.
    
    \subsubsection{Existence for $u^{n+1}$ in $L_t^\infty(\dot{H}_x^1\cap\dot{H}_x^2)$}
     We first show that \eqref{iterate} has a solution $u^{n+1} \in L_t^\infty(\dot{H}_x^1\cap\dot{H}_x^2)$ with 
    \[
    E^{n+1}(t) := \int_{\mathbb{R}}(u^{n+1}_{xx}(t))^2+(u^{n+1}_{x}(t))^2 \, dx \leq K\|u_0\|_X^2 =: E,
    \]
    for $K > 0$ a large absolute constant. We assume by induction that this is true for $u^n$.
     
     We consider the Cauchy problem
    \begin{equation}\label{integrated-eqn}\left\{\begin{aligned}
     &v_t+v_{xxx}+(u^n)_xv+u^nv_x=\dfrac{(u_x^n)^2}{2},\\
     &v(0)=(u_0)_x.
 \end{aligned}\right.     \end{equation}
    By applying Lemma $\ref{H_x^1 well-posedness}$, we obtain that \eqref{integrated-eqn} admits a unique solution $v\in L_t^\infty H_x^1$. By Sobolev embedding, we obtain that $v\in L_x^\infty$, which implies that for almost every $t\in[0,T]$, $v(t)$ is locally integrable. Then we may define
    \[
    u^{n+1}(t,x)=u_0(0, 0) + \int_0^xv(t,y)\, dy.
    \]
    
 For the energy estimate, we apply the energy estimate of Lemma~\ref{H_x^1 well-posedness} to $(u^{n + 1})_x$ with the induction hypothesis to obtain that for every $t\in[0,T]$, with $T$ chosen appropriately small depending on $C$ and $\|u_0\|_X$,
 \begin{align*}
     (E^{n+1}(t))^{1/2}&\leq e^{\frac{C}{2}\int_0^t (E^n(s))^{1/2}\, ds}\left((E^{n+1}(0))^{1/2}+\frac{C}{2}\int_0^te^{-\frac{C}{2}\int_0^s (E^n(\tau))^{1/2}\, d\tau}E^n(s)\, ds \right)\\
     &\leq e^{\frac{CTE^{1/2}}{2}}\left(\frac{E^{1/2}}{2}+\frac{CTE}{2}\right)\lesssim E^{1/2}.
 \end{align*}
 In addition, the energy estimates for $u^{n+1}$ show that it is a unique solution, hence the iteration is well-defined.

 \subsubsection{$L_x^\infty$ control for $u^{n+1}$} Applying Lemma~\ref{inftyest} and choosing $T$ appropriately small depending on $E$, we have  
    \begin{equation*}
    \begin{aligned}
     \|(u^{n + 1})_{\leq 0}\|_{L_{t,x}^\infty} &\lesssim \|(u_0)_{\leq 0}\|_{L_x^\infty} + T(\|(\partial_x^{-1}(u^n_x)^2)_{\leq 0}\|_{L_{t,x}^\infty} + \|u^n\|_{L_t^\infty W^{1,\infty}}\|u^{n + 1}\|_{L_t^\infty \dot H_x^1}) \\
     &\lesssim \|(u_0)_{\leq 0}\|_{L_x^\infty} + T(\|u_x^n\|_{L_t^\infty L_x^2}^2 + \|u^n\|_{L_t^\infty X}\|u^{n + 1}\|_{L_t^\infty \dot H_x^1}) \\
     &\lesssim \frac12 E^{1/2} + TE \lesssim E^{1/2}.
    \end{aligned}
    \end{equation*}
    
 Combined with Sobolev embedding for the high frequencies,
 \begin{align*}
     \|(u^{n+1})_{> 0}\|_{L_{t, x}^\infty}\lesssim \|u^{n+1}\|_{\dot{H}_x^1\cap\dot{H}_x^2}\lesssim E^{1/2},
 \end{align*}
 we conclude that our iteration is well-defined with the uniform bound
 \begin{align*}
     \|u^{n + 1}\|_{L_t^\infty X}\leq E^{1/2}.
 \end{align*}
 \subsection{Convergence for $u^n$}
 We shall now prove that $u^n$ is a Cauchy sequence in $L_t^\infty (L_x^\infty\cap\dot{H}_x^1)$. Let $m\geq 0$ be an arbitrary integer and $z=u^{n+2}-u^{n+1}$. In this case, $z$ satisfies
 \begin{equation}\begin{aligned}
     z_t+u^{n+1}z_x+z_{xxx}&=\dfrac{\partial_x^{-1}((u_x^{n+1})^2-(u_x^n)^2)}{2}-(u^{n+1}-u^n)u^{n+1}_x =: H
 \end{aligned}
 \end{equation}
 and thus $z_x$ satisfies
 \begin{equation}\begin{aligned}
     (z_x)_t+u_x^{n+1}z_x+u^{n+1}z_{xx}+z_{xxxx}&=H_x.
 \end{aligned}
 \end{equation}
 We estimate the source term:
 \begin{equation*}
 \begin{aligned} 
    \|H_x\|_{L_t^\infty L_x^2}&\leq \|(u_x^{n+1})^2-(u_x^n)^2\|_{L_t^\infty L_x^2}+\|(u^{n+1}-u^n)u^{n+1}_{xx}\|_{L_t^\infty L_x^2}+\|(u^{n+1}-u^n)_xu^{n+1}_x\|_{L_t^\infty L_x^2}\\
    &\leq \|u_x^{n+1}-u_x^n\|_{L_t^\infty L_x^2}\|u_x^{n+1}+u_x^n\|_{L_{t, x}^\infty}+\|u^{n+1}-u^n\|_{L_{t, x}^\infty}\|u^{n+1}_{xx}\|_{L_t^\infty L_x^2} \\
    &\quad +\|u_x^{n+1}-u_x^n\|_{L_t^\infty L_x^2}\|u^{n+1}_{x}\|_{L_{t, x}^\infty}\\
    &\lesssim E^{1/2}\|u^{n+1}-u^n\|_{L_t^\infty(L_x^\infty\cap\dot{H}_x^1)},
 \end{aligned}
 \end{equation*}
 and 
  \begin{equation*}
 \begin{aligned} 
     \|H\|_{L_x^\infty}&\lesssim \|\partial_x^{-1}((u_x^{n+1})^2-(u_x^{n})^2)\|_{L_x^\infty}+\|(u^{n+1}-u^n)u^{n+1}_x\|_{L_x^\infty} \\
     &\lesssim \|u_x^{n+1}-u_x^n\|_{L_t^\infty L_x^2}\|u_x^{n+1}+u_x^n\|_{L_t^\infty L_x^2}+\|u^{n+1}-u^n\|_{L_{t, x}^\infty}\|u_x^{n+1}\|_{L_{t, x}^\infty}\\
     &\lesssim E^{1/2}\|u^{n+1}-u^n\|_{L_t^\infty(L_x^\infty\cap\dot{H}_x^1)}.
 \end{aligned}
 \end{equation*}
 
 By applying the energy estimate provided by Lemma $\ref{L_x^2 well-posedness}$ and choosing $T$ sufficiently small, we have
 \begin{align*}
    \|z_x(t)\|_{L_x^2}&\leq e^{\frac{C}{2}\int_0^t \|(u^{n+1}(s))_x\|_{L_x^\infty}\, ds}\left(\frac{C}{2}\int_0^te^{-\frac{C}{2}\int_0^s \|(u^{n+1}(\tau))_x\|_{L_x^\infty}\, d\tau}\|H_x(s)\|_{L_x^2}\, ds \right)\\
    &\lesssim e^{\frac{TCE}{2}} TE^{1/2}\|u^{n+1}-u^n\|_{L_t^\infty(L_x^\infty\cap\dot{H}_x^1)} \\
    &\ll \|u^{n+1}-u^n\|_{L_t^\infty(L_x^\infty\cap\dot{H}_x^1)}.
 \end{align*}

 For the $L^\infty$ estimates, applying Lemma~\ref{inftyest} and choosing $T$ appropriately small depending on $E$, we have  
    \begin{equation*}
 \begin{aligned}
     \|z_{\leq 0}\|_{L_{t,x}^\infty} &\lesssim T(\| H_{\leq 0}\|_{L_{t,x}^\infty} + \|u^{n + m}\|_{L_t^\infty W^{1, \infty}}\|z\|_{L_t^\infty \dot H_x^1}) \leq \frac14 \|u^{n+1}-u^n\|_{L_t^\infty(L_x^\infty\cap\dot{H}_x^1)}.
 \end{aligned}
 \end{equation*}
For the high frequencies, we use Sobolev embedding: 
  \begin{align*}
      \|z_{> 0}\|_{L_{t, x}^\infty}\lesssim \|z\|_{L_t^\infty \dot{H}_x^1} \ll \|u^{n+1}-u^n\|_{L_t^\infty(L_x^\infty\cap\dot{H}_x^1)}.
  \end{align*} 
  
  Putting everything together, and choosing $T$ sufficiently small (depending on $R$), we get
  \begin{align*}
      \|u^{n+2}-u^{n+1}\|_{L_t^\infty(L_x^\infty\cap \dot{H}_x^1)}\leq \half \|u^{n+1}-u^n\|_{L_t^\infty (L_x^\infty\cap\dot{H}_x^1)}.
  \end{align*}
 By iterating, we get 
 \begin{align*}
\|u^{n+2}-u^{n+1}\|_{L_t^\infty(L_x^\infty\cap\dot{H}_x^1))}\leq 2^{-n-1}\|u^{1}-u^{0}\|_{L_t^\infty(L_x^\infty\cap\dot{H}_x^1))}\lesssim 2^{-n}E^{\frac{1}{2}},
\end{align*} 
which shows that $u^n$ is a fundamental sequence in $L_t^\infty(\dot{H}_x^1\cap L_x^\infty)$ converging to an element $u\in L_t^\infty(\dot{H}_x^1\cap L_x^\infty)$. In particular, $u^n_x$ converges to $u_x$ in $L_t^\infty L_x^2$. As $u^n_x$ is bounded in $L_t^\infty H_x^1$ (because $u^n$ is bounded in $L_t^\infty X$), Lemma $\ref{Regularity upgrade for limits}$ implies that $u_x\in L_t^\infty H_x^1$. Therefore, $u\in L_t^\infty X$.

 \subsection{Uniqueness}
  Let $u$ and $v$ be two solutions to \eqref{HS} with initial data $u(0)=u_0$ and $v(0)=v_0$ such that $\|u_0\|_X<R$ and $\|v_0\|_X<R$. Let $w=u-v$. Recall that we have the bounds $\|u\|_{L_t^\infty X}, \|v\|_{L_t^\infty X} \leq E^{1/2}$.
  
 In this case, $w$ satisfies
 \begin{equation}\begin{aligned}
     w_t+uw_x+w_{xxx}&=-wv_x+\dfrac{\partial_x^{-1}(w_x(u_x+v_x))}{2} =: H
 \end{aligned}
 \end{equation}
 so that $w_x$ satisfies
 \begin{equation}\begin{aligned}
     (w_x)_t+uw_{xx} + \half(u_x+v_x)w_x +w_{xxxx}&=-wv_{xx}.
 \end{aligned}
 \end{equation}

 By applying the energy estimate provided by Lemma $\ref{L_x^2 well-posedness}$ and choosing $T$ sufficiently small, we get that
 \begin{equation}\label{L_x^2-diff}
 \begin{aligned}
     \|w_x\|_{L_t^\infty L_x^2}&\leq e^{\frac{C}{2}\int_0^t \|u_x(s)\|_{L_x^\infty} + \|v_x(s)\|_{L_x^\infty} \, ds}(\|(u_0)_x-(v_0)_x\|_{L_x^2}\\
     &\quad +\frac{C}{2}\int_0^te^{-\frac{C}{2}\int_0^s \|u_x(\tau)\|_{L_x^\infty} + \|v_x(\tau)\|_{L_x^\infty}\, d\tau}\|wv_{xx}\|_{L_x^2}\, ds )\\
     &\lesssim \|(u_0)_x-(v_0)_x\|_{L_x^2}+TE^{1/2}\|w\|_{L_{t,x}^\infty}.
 \end{aligned}
 \end{equation}
 For later use, we see that formally, we also have the energy estimate of Lemma~\ref{H_x^1 well-posedness},
 \begin{equation}\label{difference-est}
  \begin{aligned}
     \|w_x\|_{L_t^\infty H^1_x}&\leq e^{\frac{C}{2}\int_0^t \|u_x(s)\|_{L_x^\infty \cap \dot H^1_x} + \|v_x(s)\|_{L_x^\infty \cap \dot H^1_x} \, ds}(\|(u_0)_x-(v_0)_x\|_{H^1_x}\\
     &\quad +\frac{C}{2}\int_0^te^{-\frac{C}{2}\int_0^s \|u_x(\tau)\|_{L_x^\infty \cap \dot H^1_x} + \|v_x(\tau)\|_{L_x^\infty \cap \dot H^1_x}\, d\tau}\|wv_{xx}\|_{H^1_x}\, ds )\\
     &\lesssim \|(u_0)_x-(v_0)_x\|_{H^1_x}+T\|v_{xx}\|_{L_t^\infty(L_x^\infty\cap\dot{H}_x^1)}\|w\|_{L_t^\infty(L_x^\infty\cap\dot{H}_x^1)}.
 \end{aligned}
 \end{equation}

 For $L^\infty$ estimates, we estimate the source term:
    \begin{equation*}
 \begin{aligned} 
\|H\|_{L_x^\infty}=\left\|-wv_x+\frac{\partial_x^{-1}(w_x(u_x+v_x))}{2}\right\|_{L_x^\infty} &\lesssim \|w_x\|_{L_t^\infty L_x^2}\|u_x+v_x\|_{L_t^\infty L_x^2} + \|w\|_{L_{t, x}^\infty}\|v_x\|_{L_{t,x}^\infty} \\
 &\lesssim \|(w_x, v_x)\|_{L_t^\infty H^{\half+}}\|w\|_{L_t^\infty(L_x^\infty\cap\dot{H}_x^1)}.
  \end{aligned}
 \end{equation*}
 Then applying Lemma~\ref{inftyest} and choosing $T$ appropriately small, we have  
    \begin{equation}\label{ptwise-diff}
 \begin{aligned}
     \|w_{\leq 0}\|_{L_{t,x}^\infty} &\lesssim \|(w(0))_{\leq 0}\|_{L_x^\infty} + T(\| H_{\leq 0}\|_{L_{t,x}^\infty} + \|u\|_{L_t^\infty W^{1,\infty}}\|w\|_{L_t^\infty \dot H_x^1}) \\
     &\lesssim \|u_0-v_0\|_{L_x^\infty}+ T\|(w_x, v_x)\|_{L_t^\infty H^{\half+}}\|w\|_{L_t^\infty(L_x^\infty\cap\dot{H}_x^1)}.
 \end{aligned}
 \end{equation}
 Moreover, by Sobolev embedding,
 \[
  \|w_{> 0}\|_{L_{t, x}^\infty}\lesssim \|w_x\|_{L_t^\infty L_x^2}.
  \]
  
   By adding this inequality, along with equations \ref{L_x^2-diff} and \ref{ptwise-diff}, we get that
 \begin{align*}
    \|w\|_{L_t^\infty (L_x^\infty\cap \dot H_x^1)}&\lesssim \|(u_0)_x-(v_0)_x\|_{L_x^2}+\|u_0-v_0\|_{L_x^\infty}+TE^{1/2}\|w\|_{L_t^\infty(L_x^\infty\cap\dot{H}_x^1)}.
 \end{align*}
 Choosing $T$ sufficiently small, we find
 \begin{equation}\label{diffest}
     \|w\|_{L_t^\infty (L_x^\infty\cap \dot H_x^1)}\lesssim \|u_0-v_0\|_{L_x^\infty \cap \dot H_x^1}
 \end{equation}
 which establishes uniqueness.
 \subsection{Continuity with respect to the initial data}
 Consider a sequence of initial data
 \[
 u_{0j}\rightarrow u_0 \in X.
 \]
  Here, since $\|u_0\|_X<R$, we may assume that $\|u_{0j}\|_X<R$ for every $j$, and the existence part implies that $u_j$ and $u$ may be defined on a common time interval $[0,T]$, with uniform bounds in $j$. Furthermore, by the Lipschitz estimate from the proof of uniqueness,
 \[
 u_j\rightarrow u \in L_t^\infty (L_x^\infty\cap\dot{H}_x^1).
 \]
 By interpolation, it follows that 
 \[
 u_j \rightarrow u \in L_t^\infty([0, T],L_x^\infty\cap\dot{H}_x^1\cap\dot{H}_x^{2-\varepsilon}).\] 
 To obtain the endpoint, we take an approach similar to the one presented in \cite{LWN}. 
 
 \
 
 We define $u_{0j}^h=(u_{0j})_{\leq h}$ and $u_0^h=(u_0)_{\leq h}$, and may assume that
 \begin{align*}
     \|u^h_{0j}\|_{X} \lesssim \|u_0\|_{X},
 \end{align*}
 so that there exists $T=T(\|u_0\|_{X})>0$ and solutions $u^h$ and $u_j^h$ that belong to $L_t^\infty X$. Further, Lemma $\ref{Persistence of regularity}$ shows that $u^h$ and $u_j^h$ belong to $L_t^\infty (X\cap\dot{H}_x^3)$. 
  As 
  \[
  \displaystyle \int_0^T\|u^h_x(s)\|_{L_x^\infty}\, ds\lesssim T\|u^h\|_{L_t^\infty(\dot{H}_x^1\cap\dot{H}_x^2)},
  \]
   we have from the energy estimate of Lemma $\ref{Persistence of regularity}$ that
  \[
  \|u^h\|_{L_t^\infty(\dot{H}_x^1\cap\dot{H}_x^3)}\lesssim \|u^h_0\|_{\dot{H}_x^1\cap\dot{H}_x^3},
  \]
  and likewise for $u_j^h$.

 We consider $H_x^1$ sharp frequency envelopes for $(u_0)_x$ and $(u_{0j})_x$, denoted by $\{c_k\}_{k\in\mathbb{Z}}$ and $\{c^j_{k}\}_{k\in\mathbb{Z}}$. As $(u_{0j})_x \rightarrow (u_0)_x$ in $H_x^1$, we can assume that $c^j_{k}\rightarrow c_k$ in $l^2$. Moreover, as in \cite{LWN}, we can choose $c_k$ having the following properties:
 
 \begin{enumerate}
     \item[a)]Uniform bounds:
     \[     \|P_k(u_0^h)_x\|_{H_x^1}\lesssim c_k     \]
     \item[b)]High frequency bounds:
     \[\|(u_0^h)_x\|_{H_x^2}\lesssim 2^hc_h\]
     \item[c)]Difference bounds:
     \[\|u_0^{h+1}-u_0^h\|_{\dot H_x^1}\lesssim 2^{-h}c_h\]
     \item[d)]Limit as $h\rightarrow\infty$:
     \[ D_x u_0^h\rightarrow D_x u_0 \in H_x^1\]
 \end{enumerate}
 and likewise for $c^j_k$. 

\

We first establish estimates for $(u - u^h)_{> 0}$ and $(u_j - u_j^h)_{> 0}$ in $L_t^\infty X$. We treat the low frequencies separately because the frequency envelopes that we are using are $\dot{H}^1\cap\dot{H}^2$-based, and don't allow us to control the $L^\infty$-component of the norm of $X$ at low frequencies. By applying the Lipschitz estimate from the proof of uniqueness, we can see that
 \begin{align*}
    \|u^{h+1}-u^h\|_{L_t^\infty(\dot{H}_x^1\cap L_x^\infty)}\lesssim \|u^{h+1}_0-u^h_0\|_{\dot{H}_x^1\cap L_x^\infty} \lesssim \|u_0^{h+1}-u_0^h\|_{\dot H^1}\lesssim 2^{-h}c_h.
 \end{align*}
 Taking the high frequencies and interpolating with the estimate
 \[
 \|u^h_{>0}\|_{L_t^\infty(X\cap\dot{H}_x^3)}\lesssim \|u^h\|_{L_t^\infty(\dot{H}_x^1\cap\dot{H}_x^3)}\lesssim \|u^h_0\|_{\dot{H}_x^1\cap\dot{H}_x^3} \lesssim 2^h c_h,
 \]
  we get that 
 \begin{align*}
     \|u^{h+1}_{>0}-u^h_{>0}\|_{L_t^\infty X}&\lesssim c_h.
 \end{align*}
 The analogous analysis and estimates hold for $u_j^h$. Moreover, as in \cite{LWN}, we get that 
 \begin{align*}
 \displaystyle \|u_{>0}-u^h_{>0}\|_{L_t^\infty X}\lesssim c_{\geq h}=\left(\sum_{\substack{k\geq h}}c_k^2\right)^{1/2}, \quad \displaystyle \|(u_j)_{>0}-(u_j^h)_{>0}\|_{L_t^\infty X}\lesssim c^j_{\geq h}=\left(\sum_{\substack{k\geq h}}(c^j_k)^2\right)^{1/2}.
 \end{align*}
 
\

 Next, we show that for fixed $h$, $\displaystyle \lim_{\substack{j\rightarrow\infty}}u_j^h=u^h$ in $L_t^\infty X([0,T]\times\mathbb{R})$. Let us write $w=u^h-u_j^h$, which by \eqref{difference-est} satisfies
  \begin{equation*}
  \begin{aligned}
     \|w_x\|_{L_t^\infty H^1_x} &\lesssim \|w_x(0)\|_{H^1_x}+T\|(u_j^h)_{xx}\|_{L_t^\infty(L_x^\infty\cap\dot{H}_x^1)}\|w\|_{L_t^\infty(L_x^\infty\cap\dot{H}_x^1)}.
 \end{aligned}
 \end{equation*}

    As $h$ is fixed, the previous discussion ensures that $\|(u_j^h)_{xx}\|_{L_t^\infty(L_x^\infty\cap\dot{H}_x^1)}$ is uniformly bounded with respect to $j$. Using as well \eqref{diffest}, we conclude that
    \begin{align*}
         \|w\|_{L_t^\infty X} &\lesssim_h \|u_0-u_{0j}\|_{X}
    \end{align*}
    as desired.
     
     \
     
     To complete the argument, we have
     \begin{align*}
        \|u_{>0}-(u_j)_{>0}\|_{L_t^\infty X}&\lesssim \|u^h -u_j^h\|_{L_t^\infty X}+ \|u_{>0}-u^h_{>0}\|_{L_t^\infty X}+\|(u_j)_{>0}-(u_j^h)_{>0}\|_{L_t^\infty X}\\
        &\lesssim \|u^h-u_j^h\|_{L_t^\infty X}+c_{\geq h}+c^j_{\geq h}
     \end{align*}
     so that fixing $h$,
     \begin{align*}
        \limsup_{\substack{j\rightarrow\infty}}\|u_{>0}-(u_j)_{>0}\|_{L_t^\infty X}\lesssim c_{\geq h} + c^j_{\leq h}
     \end{align*}
     Then letting $h$ tend to $\infty$, we get that 
     \begin{align*}
         \lim_{\substack{j\rightarrow\infty}}\|u_{>0}-(u_j)_{>0}\|_{L_t^\infty X}=0.
     \end{align*}
     
  For the low frequencies, we directly estimate
  \begin{align*}
\|u_{\leq 0}-(u_j)_{\leq 0}\|_{L_t^\infty X}\lesssim \|u_{\leq 0}-(u_j)_{\leq 0}\|_{L_t^\infty (\dot{H}_x^1\cap L_x^\infty)}\lesssim \|u-u_j\|_{L_t^\infty(\dot{H}_x^1\cap L_x^\infty)}\lesssim \|u_0-u_{0j}\|_{\dot{H}_x^1\cap L_x^\infty}
\end{align*}
As $u_{j0} \rightarrow u_j$ in $X$, it follows that
\begin{align*}
         \lim_{\substack{j\rightarrow\infty}}\|u_{\leq 0}-(u_j)_{\leq 0}\|_{L_t^\infty X}=0.
     \end{align*}
     
     Combining the low and high frequencies, we obtain $u_j \rightarrow u$ in $L_t^\infty X$.

     \subsection{Continuity in time} 
      Let $h>0$ be an arbitrary parameter, and $u^h$ solve \eqref{HS} with initial data $(u_0)_{\leq h}$. In particular,
\begin{equation}\begin{aligned}
u^h_t&=\frac{\partial_x^{-1}((u^h_x)^2)}{2}-u^hu^h_x-u^h_{xxx}.
\end{aligned}
\end{equation}
From Lemma \ref{Persistence of regularity}, we know that $u^h\in L_t^\infty (X\cap\dot{H}_x^5)$, so that the right hand side belongs to $L_t^\infty X$. Thus, $u^h\in C_t^0X$. From the previous section, we know that $u^h$ converges to $u$ in $L_t^\infty X$, hence in $C_t^0X$. This concludes the proof of Theorem~\ref{Local well-posedness}.

\section{Global well-posedness}\label{s:gwp}
 
 In this section, we prove Theorem~\ref{Global well-posedness}. Recall that 
 the dispersive Hunter-Saxton \eqref{HS} has the conserved quantities (see \cite{CI})
\begin{align*}
    E_1(t)&=\int_{\mathbb{R}}u_x(t)^2 \, dx\\
    E_2(t)&=\int_{\mathbb{R}}u_{xx}(t)^2-u(t)u_x(t)^2 \, dx.
\end{align*}

    Throughout the proof, $C>0$ shall denote a universal large constant. Consider a solution $u$ of \eqref{HS} on $[0, T)$ where $T$ is finite. We shall determine a uniform bound for $\|u(t)\|_X$.
    
We begin with the $L^\infty$ estimate. The high frequencies can be controlled by the $\dot{H}^1$ norm, which is conserved via $E_1$, but the low frequencies need to be treated separately as follows. Projecting \eqref{HS} onto frequencies less than or equal to $1$, we consider
\begin{align*}
(u_{\leq 0})_t+(uu_x)_{\leq 0}+(u_{\leq 0})_{xxx}=\frac{(\partial_x^{-1}(u_x^2))_{\leq 0}}{2}.
\end{align*}

For the transport term, write
\begin{equation*}
\begin{aligned}
(uu_x)_{\leq 0} - u_{\leq 0} (u_{\leq 0})_x &= (u_{> 0} u_x)_{\leq 0} + [ P_{\leq 0}, u_{\leq 0}] u_x \\
&= (u_{> 0} u_x)_{\leq 0} + [ P_{\leq 0}, P_0 u] u_x  + [ P_{\leq 0}, u_{< 0}] P_{0} u_x 
\end{aligned}
\end{equation*}
and estimate
\begin{equation*}
\begin{aligned}
\|(u_{> 0} u_x)_{\leq 0}\|_{L_x^\infty}&\lesssim \|u_{> 0} u_x\|_{L_x^2} \lesssim \|u_x\|_{L_x^2}.
\end{aligned}
\end{equation*}
The same estimate holds for the first commutator directly, without using the commutator structure. For the second commutator,
\[
\|[ P_{\leq 0}, u_{< 0}] P_{0} u_x  \|_{L^\infty} \lesssim  \|[ P_{\leq 0}, u_{< 0}] P_{0} u_x  \|_{L^2} \lesssim \|\D_x u_{< 0}\|_{L^\infty}\|P_0 u\|_{L^2} \lesssim \|u_x\|_{L^2}^2.
\]

Besides this, we may estimate the dispersive and source terms by
\begin{align*}
\|(u_{\leq 0})_{xxx}\|_{L_x^\infty}\lesssim \|u_x\|_{L_x^2},\quad \|(\partial_x^{-1}(u_x^2))_{\leq 0}\|_{L_x^\infty}\lesssim \|u_{x}\|_{L_x^2}^2.
\end{align*}
Therefore, denoting
\[
F = \frac{(\partial_x^{-1}(u_x^2))_{\leq 0}}{2}-((uu_x)_{\leq 0} - u_{\leq 0} (u_{\leq 0})_x)-(u_{\leq 0})_{xxx},
\]
we have
\begin{align*}
\|F\|_{L_x^\infty}&\lesssim \|u_{x}\|_{L_x^2}^2+\|u_x\|_{L_x^2}=E_1 + E_1^{1/2}.
\end{align*}

As $u\in C_t^0X([0,T)\times\mathbb{R})$, we see that $u$ is continuous with respect to $t$ and $x$, and Lipschitz with respect to $x$, uniformly in $t$.
 As in \cite{YY}, let us consider the flow
\begin{align*}
q_t&=u_{\leq 0}(t,q(t,x)), \qquad q(0,x)=x.
\end{align*}
By standard ordinary differential equations theory, $q$ exists, is unique, and is defined on the whole interval $[0,T)$ as a function in $C^1([0,T))$. Moreover, it is not difficult to see that it is a $C^1$-diffeomorphism. We also note that $q_{xt}=u_xq_x$, which means that $\displaystyle q_x=e^{\int_0^t u_x(s,q(s,x))\, ds}>0$, hence $q$ is strictly increasing in $x$ for every $t$. Further,
\begin{align*}
\frac{d}{dt}u_{\leq 0}(t,q(t,x))&=(u_{\leq 0})_t+u_{\leq 0}(u_{\leq 0})_x=F.
\end{align*}
Then
\begin{align*}
\|u_{\leq 0}(t,q(t,x))\|_{L_x^\infty}&\lesssim \|u_{\leq 0}(0)\|_{L_x^\infty}+\int_0^t \|F\|_{L_x^\infty} \, ds \lesssim \|(u_0)_{\leq 0}\|_{L_x^\infty}+\int_0^tE_1+E_1^{1/2}\, ds.
\end{align*}
As $q$ is a diffeomorphism, we now infer that
\begin{align*}
\|(u(t))_{\leq 0}\|_{L_x^\infty}\lesssim\|(u_0)_{\leq 0}\|_{L_x^\infty}+t(E_1+E_1^{1/2}).
\end{align*}
For the high frequencies, we apply Sobolev embeddings and Bernstein's inequalities to estimate
\begin{align*}
\|(u(t))_{> 0}\|_{L_x^\infty}\lesssim E_1^{1/2}.
\end{align*}
Combining these estimates, we conclude that for every $t\in[0,T)$,
\begin{align*}
    \|u(t)\|_{L_x^\infty}&\lesssim \|u_0\|_{X^0}+t(E_1+E_1^{1/2}).
\end{align*}
Thus, for some constant $C>0$, and for every $t\in[0,T)$, we have
\begin{align*}
\|u_{xx}(t)\|^2_{L_x^2}&\lesssim \vert E_2\vert+ \|u(t)\|_{L_x^\infty}\|u_x(t)\|^2_{L_x^2}\\
&\lesssim \|u_0\|_{\dot H_x^2}^2+\|u_0\|_{X^0}E_1 + t(E_1+E_1^{1/2})E_1.
\end{align*}
We obtain the desired estimate for $\|u(t)\|_X$, where $t\in[0,T)$. In particular, the lifespan for $u$ may be extended indefinitely.

\section{A normal form analysis}\label{s:nf}

In this section, we use normal forms to construct an energy functional corresponding to $\dot H^{1 + s}$. Since \eqref{HS} exhibits a quasilinear behavior at low frequencies, we use a modified energy approach as introduced in \cite{Burgers-Hilbert}.

We may re-express the dispersive Hunter-Saxton \eqref{HS} as
\[
u_t+u_{xxx}= \partial_x^{-2}(u_xu_{xx}) - uu_x =: Q_1 + Q_2 =: Q.
\]
Thus we see that the formal normal form variable, based on the normal form correction for the KdV equation, is
\[
\tilde u = u + B(u, u) = u - \frac16 \D_x^{-2} (u^2) + \frac16 (\D_x^{-1} u)^2.
\]

To construct a modified energy for $\dot H^{1 + s}$, write
\[
A(D) = D^s P_{> 0}
\]
and consider
\[
\int A u_x \cdot A\left(u_x - \frac13 \D_x^{-1} (u^2) + \frac23 (\D_x^{-1} u) u\right) \, dx.
\]
Integrating by parts on the last two terms and rearranging, we obtain
\[
\int (A u_x)^2 - \frac13 A u \cdot A\left( u^2 + 2 \D_x^{-1} u \cdot u_x\right) \, dx.
\]
Then commuting $A$ through the last term, we have
\[
\int (A u_x)^2 - \frac13 A u \cdot ( A( u^2) + 2 [A, \D_x^{-1} u ] u_x +2 \D_x^{-1}u \cdot  A u_x ) \, dx.
\]
Lastly, integrating by parts on the last term, we define the modified energy
\[
\tilde E(t) := \int (A u_x)^2 -\frac13 Au \cdot (A(u^2) + 2 [A, \D_x^{-1} u ] u_x - Au \cdot u) \, dx.
\]
\begin{lemma}\label{Modified energy estimates}
If $u\in C_t^0X^s([0,T)\times\mathbb{R})$, then for every $t\in [0,T)$, we have
\begin{align*}
\|(u(t))_{>0}\|_{\dot H^{1 + s}_x}^2 = \tilde E(t) + O(E_1\|u(t)\|_{L_x^\infty}),
\end{align*}
and
\[
 \frac{d}{dt} \tilde E(t) \lesssim \|Au\|_{L_x^2}^2(\|u_x\|_{L_x^2}^2 + \|u_x\|_{L_x^\infty} \|u\|_{L_x^\infty}).\]
\end{lemma}
\begin{proof}
We have
\begin{equation*}
\begin{aligned}
\|[A, \D_x^{-1} v ] w_x\|_{L_x^2}&\lesssim \|Aw\|_{L_x^2}\|v\|_{L_x^\infty}+\|Av\|_{L_x^2}\|w\|_{L_x^\infty}, \\
\|A(vw)\|_{L_x^2}&\lesssim \|Av\|_{L_x^2}\|w\|_{L_x^\infty}+\|Aw\|_{L_x^2}\|v\|_{L_x^\infty}.
\end{aligned}
\end{equation*}
Thus, the first bound is immediate. 

\

We now prove the energy estimate. First observe that $\frac{d}{dt}\tilde E$ consists only of quartic terms. Precisely, if we set
\[
L_A(v, w) := -\frac13 A(vw) - \frac23 [A, \D_x^{-1} v ] w_x + \frac13 Av  \cdot w,
\]
then a straightforward computation shows that
\[
\frac{d}{dt} \tilde E = \int AQ\cdot L_A(u, u) + Au \cdot L_A(Q, u) + Au \cdot L_A(u, Q) \, dx.
\]

We consider first the contribution from $Q_1$. Since
\[
\|L_A(v, w)\|_{L^2_x} \lesssim \|Av\|_{L_x^2}\|w\|_{L_x^\infty}+\|Aw\|_{L_x^2}\|v\|_{L_x^\infty},
\]
we have 
\begin{align*}
\int AQ_1\cdot L_A(u, u) + Au \cdot L_A(Q_1, u) &+ Au \cdot L_A(u, Q_1) \, dx \\
&\lesssim \|Au\|_{L_x^2}(\|AQ_1\|_{L_x^2}\|u\|_{L_x^\infty}+\|Au\|_{L_x^2}\|Q_1\|_{L_x^\infty}).
\end{align*}
To bound $Q_1$, we have
\begin{align*}
\|\partial_x^{-1}(u_x^2)\|_{L_x^\infty}&\lesssim \|u_x\|_{L_x^2}^2, \\
\|A\partial_x^{-1}(u_x^2)\|_{L_x^2}&\lesssim \|u_x\|_{L_x^\infty}\|Au\|_{L_x^2}
\end{align*}
which suffices.

\

For the contribution from $Q_2$, we consider each of the three terms in 
\[
L_A(u, u) = -\frac13 A(u^2) - \frac23 [A, \D_x^{-1} u ] u_x + \frac13 Au  \cdot u
\]
successively. From the third term, and the $Q_2$ contribution arising from the case where the time derivative falls on the lone $u$, 
\[
\int \frac13 Au \cdot Au \cdot uu_x \, dx \lesssim \|u\|_{L^\infty} \|u_x\|_{L^\infty} \|Au\|_{L^2}^2.
\]
On the other hand, when the derivative falls on $Au$, we write
\[
\int \frac16 Au \cdot A\D_x (u^2) \cdot u \, dx = \int \frac13 Au \cdot [A, u] u_x \cdot u + \frac13 Au \cdot u \cdot A u_x \cdot u \, dx.
\]
The latter term is the same as the previous case after an integration by parts, while 
\[
\int \frac13 Au \cdot [A, u] u_x \cdot u \, dx \lesssim \|Au\|_{L^2}^2 \|u_x\|_{L^\infty} \|u\|_{L^\infty}.
\]

From the first term in $L_A$, the case when the time derivative falls on $Au$ vanishes via an integration by parts. Then from the remaining contribution,
\[
\int  \frac13 Au \cdot A\D_x (u^3) \, dx = \int Au \cdot [A, u^2] u_x + Au \cdot u^2 \cdot A u_x \, dx.
\]
The latter term has already appeared, while 
\[
\int Au \cdot [A, u^2] u_x \, dx \lesssim \|Au\|_{L^2}^2 \|u_x\|_{L^\infty} \|u\|_{L^\infty}.
\]

Lastly, we have the commutator term from $L$. When the time derivative falls inside the commutator, we have
\[
\int Au\cdot [A,\partial_x^{-1}(uu_x)]u_x\,  dx\lesssim \|Au\|_{L^2}\|[A,\partial_x^{-1}(uu_x)]u_x\|_{L^2}\lesssim \|Au\|_{L^2}^2 \|u_x\|_{L^\infty} \|u\|_{L^\infty}.
\]
From the remaining contributions of $Q_2$, we are left with 
\begin{align*}
    \int A(uu_x)\cdot [A,\partial_x^{-1}u]u_x+Au\cdot[A,\partial_x^{-1}u](uu_x)_x\,  dx.
\end{align*}
Integrating by parts on the second term, and since 
\[
\int Au\cdot[A,u](uu_x)\,  dx\lesssim \|Au\|_{L^2}\|u_x\|_{L_x^\infty}\|A(u^2)\|_{L^2}\lesssim \|Au\|_{L^2}^2 \|u_x\|_{L^\infty} \|u\|_{L^\infty},
\]
it remains to bound
\begin{equation}\label{dcom}
\int A(uu_x)\cdot [A,\partial_x^{-1}u]u_x-Au_x\cdot[A,\partial_x^{-1}u](uu_x)\,  dx = -\int u_x\cdot [A[A,\partial_x^{-1}u],u]u_x\,  dx.
\end{equation}

Before exploiting the full commutator structure, we first reduce to paraproducts. 

From the first integral on the left hand side of \eqref{dcom}, we write
\begin{align*}
   \int A(uu_x)\cdot [A,\partial_x^{-1}u]u_x\, dx&= \int A(uu_x)\cdot[A,T_{\partial_x^{-1}u}]u_x\,  dx\\
   &\quad -\frac{1}{2}\int A(u^2)\cdot \partial_x(A(T_{u_x}\partial_x^{-1}u)+A\Pi(u_x,\partial_x^{-1}u))\,  dx\\
   &\quad +\half \int A(u^2)\cdot \partial_x(T_{Au_x}\partial_x^{-1}u+\Pi(Au_x,\partial_x^{-1}u))\,  dx.
\end{align*}
The last two lines are perturbative and may be discarded. Precisely, we have
\[
\|A(u^2)\|_{L_x^2}\lesssim\| u\|_{L_x^\infty}\|Au\|_{L_x^2}
 \]
while 
\begin{align*}
    \|\D_xA(T_{u_x}\partial_x^{-1}u)\|_{L_x^2}&\lesssim \|u_x\|_{L_x^\infty}\|Au\|_{L_x^2},\\
    \|\D_x(T_{Au_x}\partial_x^{-1}u)\|_{L_x^2}&\lesssim \|u_x\|_{L_x^\infty}\|Au\|_{L_x^2},
\end{align*}
with the same estimate for the balanced frequency terms. 

Next, we proceed further to write 
\begin{align*}
   \int A(uu_x)\cdot [A,T_{\partial_x^{-1}u}]u_x\, dx &= \int A(T_uu_x)\cdot [A,T_{\partial_x^{-1}u}]u_x\, dx\\
   &\quad +\int A(T_{u_x} u)\cdot [A,T_{\partial_x^{-1}u}]u_x\, dx+\int A\Pi(u_x, u)\cdot [A,T_{\partial_x^{-1}u}]u_x\, dx.
\end{align*}
The second line is perturbative as before. Precisely,
\[
\int A(T_{u_x} u)\cdot [A,T_{\partial_x^{-1}u}]u_x\, dx \lesssim \|u_x\|_{L^\infty_x} \|Au\|_{L^2_x} \cdot \|u\|_{L^\infty_x}\|Au\|_{L^2_x}
\]
with the same estimate for the balanced frequency term.

A similar analysis holds for the second term on the left hand side of \eqref{dcom}, so we are only left to estimate
\begin{align*}
  \int A(T_uu_x)\cdot[A,T_{\partial_x^{-1}u}]u_x\, dx-  \int Au_x\cdot[A,T_{\partial_x^{-1}u}](T_uu_x)\,  dx =  -\int u_x\cdot [A[A,T_{\partial_x^{-1}u}],T_u]u_x\, dx.
\end{align*}

\

Define
\[
L(u, v, w) = D^{-s}\D_x [A[A, T_{\D_x^{-1} u}], T_{\D_x^{-1} v}] D^{-s} w_{x}
\]
and let $L_k$ denote the frequency $k$ component.

%This can be rewritten as
%\begin{align*}
 %   \int u_x\cdot [A[A,T_{\partial_x^{-1}u}],T_u]u_x\, dx=\sum_{\substack{k}} \int (u_k)_x\cdot [A[A,\partial_x^{-1}u_{<k}],u_{<k}](u_k)_x\, dx
%\end{align*}
%Let us define $\displaystyle L_k(f,g,h)=[A[A,\partial_x^{-1}f_{<k}],\partial_x^{-1}g_{<k}](h_k)_x$. 
%We have
%\begin{align*}
%    L_k(f,g,h)&=[A[A,\partial_x^{-1}f_{<k}],\partial_x^{-1}g_{<k}](h_k)_x=A[A,\partial_x^{-1}f_{<k}](\partial_x^{-1}g_{<k}(h_k)_x)-\partial_x^{-1}g_{<k}A[A,\partial_x^{-1}f_{<k}](h_k)_x\\
%    &=A(A(\partial_x^{-1}f_{<k}\partial_x^{-1}g_{<k}(h_k)_x)-\partial_x^{-1}f_{<k}A(\partial_x^{-1}g_{<k}(h_k)_x))\\
%    &-\partial_x^{-1}g_{<k}A(A(\partial_x^{-1}f_{<k}(h_k)_x)-\partial_x^{-1}f_{<k}A(h_k)_x)
%\end{align*}
Let $\displaystyle a(\xi)=\vert\xi\vert^s(1-\phi(\xi))$ be the symbol of $A$, where $\phi$ is the symbol of the Littlewood-Paley projector $P_{\leq 0}$, and
\[
 \phi_k(\xi)=\phi\left(\frac{\xi}{2^{k-4}}\right), \qquad \psi_k(\xi)=\phi\left(\frac{\xi}{2^k}\right)-\phi\left(\frac{\xi}{2^{k-1}}\right).
\]
The symbol of $L_k$ is
\begin{align*}
    L_k(\xi,\eta,\zeta)&=\phi_k(\xi)\phi_k(\eta)\zeta\psi_k(\zeta) (\xi\eta\vert\xi+\eta+\zeta \vert^s\vert\zeta\vert^s)^{-1} \\
    &\quad \cdot \left(a(\xi+\eta+\zeta)(a(\xi+\eta+\zeta)-a(\eta+\zeta))-a(\xi+\zeta)(a(\xi+\zeta)-a(\zeta))\right).
\end{align*}
This symbol is supported in the region $\displaystyle\{(\xi,\eta,\zeta)\vert\xi,\eta\lesssim 2^k,\zeta\sim 2^k\}$, is smooth, and its associated kernel is bounded and integrable. Thus, we have the estimate:
\begin{align*}
   -\int (u_k)_x\cdot [A[A,\partial_x^{-1}u_{<k}],u_{<k}](u_k)_x\, dx&=\int Au_k \cdot L_k(u,u_x,Au_k) \, dx \\
   &\lesssim \|u\|_{L_x^\infty}\|u_x\|_{L_x^\infty}\|Au_k\|^2_{L_x^2}.
\end{align*}
Thus,
\begin{align*}
    \int u_x\cdot [A[A,T_{\partial_x^{-1}u}],T_u]u_x\, dx\lesssim \sum_{\substack{k}}\|u\|_{L_x^\infty}\|u_x\|_{L_x^\infty}\|Au_k\|^2_{L_x^2}\lesssim \|u\|_{L_x^\infty}\|u_x\|_{L_x^\infty}\|Au\|^2_{L_x^2}
\end{align*}
By putting everything together, we obtain the desired estimate.

\end{proof}

By combining the previous result with the $L_x^\infty$ bounds from Theorem \ref{Global well-posedness}, we establish bounds on the growth of the solutions in $\dot H^{1 + s}$.

\begin{proposition}\label{Uniform bounds}
Let $T > 0$, $I = [0,T]$ or $I=[0,T)$, and $u\in C_t^0X^s(I\times\mathbb{R})$ solve \eqref{HS}. Then we have the bounds \eqref{low-reg-growth-bounds}.
\end{proposition}
\begin{proof}
We have from Theorem \ref{Global well-posedness} the pointwise estimates. It remains to establish the energy bounds.

Let $\tilde{E}$ be the modified energy functional of Lemma~\ref{Modified energy estimates}, so that for $t \in [0, T)$,
\begin{align*}
 \frac{d}{dt} \tilde E(t) &\lesssim \|Au\|_{L_x^2}^2(\|u_x\|_{L_x^2}^2 + \|u_x\|_{L_x^\infty} \|u\|_{L_x^\infty}) \\
 &\lesssim \|u_x(t)\|^4_{L_x^2} +\|u_x(t)\|^3_{L_x^2}\|u(t)\|_{L_x^\infty}+ \|u_x(t)\|^2_{L_x^2}\|u(t)\|_{L_x^\infty}\|u(t)_{>0}\|_{\dot H_x^{1 + s}}
\end{align*}
and since we have the energy equivalence
\begin{align*}
\|(u(t))_{>0}\|_{\dot H^{1 + s}_x}^2 = \tilde E(t) + O(E_1\|u(t)\|_{L_x^\infty}),
\end{align*}
we find 
\begin{align*}
 \frac{d}{dt} \tilde E(t) &\lesssim E_1^2 +E_1^{3/2}\|u(t)\|_{L_x^\infty}+ E_1\|u(t)\|_{L_x^\infty}(\tilde E(t) +  CE_1\|u(t)\|_{L_x^\infty})^{1/2} \\
 &\lesssim E_1^{2} + E_1\|u(t)\|_{L_x^\infty}(\tilde E(t)+E_1+CE_1\|u(t)\|_{L_x^\infty})^{1/2} \\
  &\lesssim E_1\|u(t)\|_{X^0}(\tilde E(t)+E_1+CE_1\|u(t)\|_{L_x^\infty})^{1/2} \\
  &\lesssim E_1\|u\|_{L_t^\infty X^0}(\tilde E(t)+E_1+CE_1\|u\|_{L_{t,x}^\infty})^{1/2}.
\end{align*}
Integrating in $t$, we find that for every $t\in[0,T)$,
\begin{align*}
\left(CE_1\|u\|_{L_{t,x}^\infty} + E_1 + \tilde E(t)\right)^{\frac{1}{2}}&\lesssim tE_1\|u\|_{L_t^\infty X^0}+\left(CE_1\|u\|_{L_{t,x}^\infty} + E_1 + \tilde E(0)\right)^{\frac{1}{2}}.
\end{align*}
Thus,
\begin{align*}
    \tilde{E}(t)&\lesssim t^2E_1^2\|u\|_{L_t^\infty X^0}^2+E_1(\|u\|_{L_{t,x}^\infty} + 1) + \tilde E(0).
\end{align*}
Using the first inequality from Lemma \ref{Modified energy estimates} and the low frequency bound
\[
\|(u(t))_{\leq 0}\|_{\dot H_x^{1+s}}^2\lesssim E_1,
\]
we have 
\begin{align*}
     \|u(t)\|_{\dot H_x^{1+s}}^2 &\lesssim t^2E_1^2\|u\|_{L_t^\infty X^0}^2 +E_1(\|u\|_{L_{t,x}^\infty} + 1) +\|u_0\|^2_{\dot H_x^{1+s}}.
\end{align*}
Combined with the pointwise estimates, we obtain the stated bound.
\end{proof}

This establishes the bounds in Theorem \ref{Low regularity global well-posedness}. Global well-posedness now follows from the local result of Theorem \ref{Low regularity local well-posedness}, which we prove in the next two sections.

\section{An estimate for the linearized equation}\label{s:lin}

The linearized equation corresponding to \eqref{HS} is
 \begin{equation*}
 \begin{aligned}
     w_t+(uw)_x+w_{xxx}=\D_x^{-1}(u_xw_x),
 \end{aligned}
 \end{equation*}
 which can be rewritten as 
 \begin{equation}\label{paralinHS}
 \begin{aligned}
     w_t+ \D_x^{-1} (u_{xx} w) + uw_x +w_{xxx}= f.
 \end{aligned}
 \end{equation}
  Applying $D^s$ with $s \in (\half, 1)$ to \eqref{paralinHS}, and writing $v = D^s w$, we have
   \begin{equation}\label{paralinHS-half}
   \begin{aligned}
     v_t + u v_x +v_{xxx}&= - D^s\D_x^{-1}(T_{u_{xx}}w + T_w u_{xx} + \Pi(w, u_{xx})) - [D^s, u] w_x + D_x f.
     \end{aligned}
 \end{equation}
 
 \begin{lemma}\label{lin-full}
 Let $T > 0$ and $I = [0, T]$. If $u \in L_t^\infty X^s$ is a solution of \eqref{HS} in $I$  and $w \in C_t^0(L_x^\infty\cap\dot H_x^s)(I\times\mathbb{R})$ is a solution of \eqref{paralinHS}, then by shrinking $T$ enough depending on $\|u\|_{L_t^\infty X^s(I\times\mathbb{R})}$, we have
 \begin{align*}
 \|w\|_{L_t^\infty(L_x^\infty\cap\dot H_x^s)}&\lesssim\|w_0\|_{L_x^\infty\cap\dot H_x^s} + \|f\|_{L^1_t (L^\infty_x \cap \dot H^s_x)}.
\end{align*}

 \end{lemma}
 \begin{proof}

We consider the homogeneous problem with $f = 0$, as the proof below easily generalizes. We first bound the source terms of \eqref{paralinHS-half} in $L^2$. For the first two source terms, we have
  \[
  \| D^s\D_x^{-1}(T_{u_{xx}}w)\|_{L_x^2} \lesssim \|w\|_{\dot H_x^s} \|u_x\|_{L_x^\infty}\lesssim \|w\|_{\dot H_x^s} \|u\|_{\dot H_x^{1}\cap\dot H_x^{1+s}}
 \]
 and
 \[
  \| D^s\D_x^{-1}(T_w u_{xx})\|_{L_x^2} \lesssim \|w\|_{L_x^\infty} \|u\|_{\dot H_x^{1 + s}}.
 \]
 For the balanced frequency case, we have
 \begin{align*}
     \| D_x^s\D_x^{-1}\Pi(w,u_{xx})\|_{L_x^2}&\lesssim \| D^{1}\D_x^{-1}\Pi(w,u_{xx})\|_{L_x^{\frac{2}{3-2s}}}\lesssim  \|D_x^{s}w\|_{L_x^{2}} \|D_x^{2-s}u\|_{L_x^{\frac{1}{1-s}}}\\
     &\lesssim \|D_x^{s}w\|_{L_x^{2}} \|D_x^{\frac{3}{2}}u\|_{L_x^{2}}\lesssim\|w\|_{\dot H_x^s} \|u\|_{\dot H_x^{1}\cap\dot H_x^{1+s}}.
 \end{align*}
  Lastly, for the commutator term, we have
 \[
 \|[D^s, u] w_x\|_{L_x^2} \lesssim \|u_x\|_{L_x^\infty} \|w\|_{\dot H^s_x} .
 \]
 
By applying the energy estimate from Lemma \ref{L_x^2 well-posedness}, we get that for every $t\in [0,T]$,
\begin{equation*}\begin{aligned}
 \|w(t)\|_{\dot H_x^s}&\leq e^{\frac{C}{2}\int_0^t\|u_x(\tau)\|_{L_x^\infty}\, d\tau}\left(\|w_0\|_{\dot H_x^s}+\int_0^te^{-\int_0^\tau\|u_x(\eta)\|_{L_x^\infty}d\eta}\|u(\tau)\|_{\dot H_x^1\cap\dot H_x^{1 + s}}\|w(\tau)\|_{L_x^\infty\cap\dot H_x^s}\, d\tau\right)\\
 &\lesssim \|w_0\|_{\dot H_x^s}+ T\|u\|_{L_t^\infty X^s}\|w\|_{L_x^\infty\cap\dot H_x^s}.
\end{aligned}\end{equation*}

Next, to obtain an $L^\infty$ estimate, it suffices to consider the low frequencies since by Sobolev embedding,
  \begin{align*}
      \|w_{> 0}\|_{L_{x}^\infty}\lesssim \|w\|_{\dot H^s_x}.
  \end{align*} 
  For the first source term, we decompose into paraproducts as before to estimate
   \begin{align*}
 \| P_{\leq 0}\D_x^{-1}(T_{u_{xx}} w)\|_{L_x^\infty} &\lesssim \| P_{\leq 0}\D_x^{-1}(T_{u_{xx}} w)\|_{L_x^{\frac{1}{1-s}}}\lesssim \|D_x^{-\frac{s}{2}}u_x\|_{L_x^{\frac{2}{1-s}}}\|D_x^{\frac{s}{2}}w\|_{L_x^{\frac{2}{1-s}}}\\
  &\lesssim \|u_x\|_{L_x^2}\|w\|_{\dot H_x^s} \lesssim \|w\|_{\dot H_x^s} \|u\|_{\dot H_x^1\cap\dot H_x^{1+s}}, \\
 \|P_{\leq 0} \D_x^{-1} (T_{w} u_{xx})\|_{L_x^\infty} &\lesssim \| \D_x^{-1}(T_{w}u_{xx})\|_{H_x^s}\lesssim \|w\|_{L_x^\infty} \|u\|_{\dot H_x^1\cap\dot H_x^{1+s}},
  \end{align*}
 and
 \[
 \|P_{\leq 0} \D_x^{-1} \Pi(u_{xx}, w)\|_{L_x^\infty} \lesssim  \|P_{\leq 0} \Pi(u_{xx}, w)\|_{L_x^1} \lesssim \|u_x\|_{\dot H^{1-s}_x} \|w\|_{\dot H^s_x}\lesssim \|u\|_{\dot H^{1}_x\cap\dot H^{1+s}_x} \|w\|_{\dot H^s_x}.
 \]
 Thus, 
 \[
 \|P_{\leq 0} \D_x^{-1} (w u_{xx})\|_{L_x^\infty}\lesssim \|w\|_{L_x^\infty} \|u\|_{\dot H^{1}_x\cap\dot H^{1+s}_x} + \|u\|_{\dot H^{1}_x\cap\dot H^{1+s}_x} \|w\|_{\dot H^s_x}.
 \]
From Lemma \ref{inftyest}, with $T'$ sufficiently small, we have
  \begin{equation*}
 \begin{aligned}
     \|w_{\leq 0}\|_{L_{t,x}^\infty} \lesssim \|(w_0)_{\leq 0}\|_{L_x^\infty} + T'\|u\|_{L_t^\infty X^s}\|w\|_{L_t^\infty \dot H_x^{s}}.
 \end{aligned}
 \end{equation*}
  
  Putting everything together, for $t\in [0,T']$ we get
  \begin{align*}
 \|w(t)\|_{L_x^\infty\cap\dot H_x^s}&\lesssim \|w_0\|_{L_x^\infty\cap\dot H_x^s}+T'\|u\|_{L_t^\infty X^s}\|w\|_{L_x^\infty\cap\dot H_x^s}.
\end{align*}
By further shrinking $T'$ depending on $\displaystyle\|u\|_{L_t^\infty X^s}$, we obtain the desired estimate.
 \end{proof}
We now prove a result regarding differences of solutions, that is going to be used in order to justify uniqueness of $C_t^0X^s$-solutions in the proof of Theorem \ref{Low regularity local well-posedness}.

\begin{lemma}\label{Regularity for differences}
 Let $T > 0$ and $I = [0, T]$. Let $u,v\in C_t^0X^s(I\times\mathbb{R})$ solve \eqref{HS} with $u_0-v_0\in L_x^\infty\cap\dot H_x^s$. Then $u-v\in L_t^\infty(L_x^\infty\cap\dot H_x^s)(I\times\mathbb{R})$, and for $T$ sufficiently small depending on $\|(u, v)\|_{L_t^\infty X^s}$, we have
 \begin{align*}
 \|u-v\|_{L_t^\infty(L_x^\infty\cap\dot H_x^s)}&\lesssim \|u_0-v_0\|_{L_x^\infty\cap\dot H_x^s}.
 \end{align*}
\end{lemma}
\begin{proof}
Let $z=u-v$, which solves the equation
\begin{equation}\begin{aligned}\label{dif}
    z_t+uz_x+v_xz+z_{xxx}=\frac{\partial_x^{-1}(z_x(u_x+v_x))}{2}.
\end{aligned}\end{equation}
We apply $\displaystyle D_x^s$ and rearrange to consider the Cauchy problem
\begin{align*}
    w_t+w_{xxx}=D_x^s\left(\frac{\partial_x^{-1}(z_x(u_x+v_x))}{2}\right)-D_x^s(uz)_x+D_x^s(zz_x):=H
\end{align*}
with initial data $w(0)=D_x^s(u_0-v_0)\in L_x^2(\mathbb{R})$. For the first term in $H$, we decompose into paraproducts and have the $L^2$ bounds
 \begin{align*} 
  \| D_x^s\D_x^{-1}(T_{z_{x}} (u_x+v_x))\|_{L_x^2} &\lesssim \|z\|_{L_x^\infty} \|u+v\|_{\dot H_x^{1+s}} \\
  \| D_x^s\D_x^{-1}(T_{u_x+v_x} z_{x})\|_{L_x^2} &\lesssim \|z\|_{\dot H_x^{1+s}} \|u+v\|_{L_x^\infty} \\
 \| D_x^s\D_x^{-1}\Pi(u_x+v_x,z_{x})\|_{L_x^2} &\lesssim \|D_x\D_x^{-1}\Pi(u_x+v_x,z_{x})\|_{L_x^{\frac{2}{3-2s}}}\\
 &\lesssim  \|z\|_{\dot H_x^{1}} \|D_x(u+v)\|_{L_x^{\frac{1}{1-s}}}\lesssim \|z\|_{\dot H_x^{1}} \|D_x^{s+1/2}(u+v)\|_{L_x^2}.
 \end{align*}
The other terms are estimated directly using product estimates: 
\begin{equation*}
\begin{aligned}
 \|D_x^s\D_x(uz)\|_{L_x^2}&\lesssim \|u\|_{L_x^\infty}\|z\|_{\dot H_x^{1+s}} + \|z\|_{L_x^\infty}\|u\|_{\dot H_x^{1+s}}\\
 \|D_x^s\D_x (z^2)\|_{L_x^2}&\lesssim \|z\|_{L_x^\infty}\|z\|_{\dot H_x^{1+s}}.
\end{aligned}
\end{equation*}
 Thus, $\displaystyle H\in L_t^\infty L_x^2([0,T]\times\mathbb{R})$. By applying Lemma \ref{L_x^2 well-posedness}, we infer that \eqref{dif} has a unique solution in $L_t^\infty L_x^2([0,T]\times\mathbb{R})$. However, both $w$ and $D_x^sz$ are solutions (in the sense of tempered distributions), hence $w=D_x^sz$, and $z=u-v\in L_t^\infty \dot H_x^s([0,T]\times\mathbb{R})$. It is also clear that $u-v\in L_{t,x}^\infty([0,T]\times\mathbb{R})$.
 
 We now observe that $z$ satisfies the linearized equation \eqref{paralinHS} with source,
 \begin{align*}
     z_t+\partial_x^{-1}(u_{xx}z) + uz_x+z_{xxx}&=zz_x-\frac{\partial_x^{-1}(z_x^2)}{2} =: f.
 \end{align*}
 After taking $T$ small enough (depending on $\displaystyle\|u\|_{L_t^\infty X^s}$), we can apply Lemma~\ref{lin-full}, But first, we have to estimate $f \in L^1_t (L^\infty_x \cap \dot H^s_x)([0,T]\times\mathbb{R})$. For the first term of $f$,
 \begin{align*}
     \|D_x^s\D_x (z^2)\|_{L_x^2}&\lesssim \|z\|_{L_x^\infty}(\|u\|_{X^s}+\|v\|_{X^s}), \\
 \|\D_x (z^2)\|_{L_x^\infty}&\lesssim \|z\|_{L_x^\infty}(\|u\|_{X^s}+\|v\|_{X^s}).
 \end{align*}
 For the second, we have
 \begin{align*}
     \|D_x^s\D_x^{-1}(T_{z_x}z_x)\|_{L_x^2}&\lesssim \|z\|_{\dot H_x^s}\|z_x\|_{L_x^\infty}\lesssim \|z\|_{\dot H_x^s}(\|u\|_{X^s}+\|v\|_{X^s})\\
     \|D_x^s\D_x^{-1}\Pi(z_x,z_x)\|_{L_x^2}&\lesssim \|\Pi(z_x,z_x)\|_{L_x^{\frac{2}{3-2s}}}\lesssim \|D_x^sz\|_{L_x^2}\|D_x^{2-s}z\|_{L_x^{\frac{1}{1-s}}}\lesssim \|z\|_{\dot H_x^s}\|D_x^{3/2}z\|_{L_x^2}\\
     &\lesssim \|z\|_{\dot H_x^s}(\|u\|_{X^s}+\|v\|_{X^s}).
 \end{align*}
 For the $L_x^\infty$ estimate, it suffices to consider the low frequencies since by Sobolev embedding,
  \begin{align*}
      \|\D_x^{-1}(z_x^2)_{> 0}\|_{L_{x}^\infty}\lesssim \|\D_x^{-1}(z_x^2)_{> 0}\|_{\dot H^s_x}.
  \end{align*} 
   We then have for the low frequencies
 \begin{align*}
     \|P_{\leq 0}\D_x^{-1}(\Pi({z_{x}},z_x))\|_{L_x^\infty}& \lesssim \|P_{\leq 0}\Pi(z_{x},z_x)\|_{L_x^1}\lesssim \|z\|_{\dot H_x^s} \|z\|_{\dot H_x^{2-s}} \lesssim \|z\|_{\dot H_x^s}(\|u\|_{X^s}+\|v\|_{X^s}) \\     
  \| P_{\leq 0}\D_x^{-1}(T_{z_{x}}z_x)\|_{L_x^\infty} &\lesssim \| \D_x^{-1}(T_{z_{x}}z_x)\|_{H_x^s}\lesssim \|z\|_{L_x^\infty}(\|u\|_{X^s}+\|v\|_{X^s}).
 \end{align*}
 Thus,
 \[
 \|f\|_{L^1_t (L^\infty_x \cap \dot H^s_x)}\lesssim T\|z\|_{L_t^\infty(L_x^\infty\cap\dot H_x^s)}(\|u\|_{L_t^\infty X^s}+\|v\|_{L_t^\infty X^s}).
 \]
 Thus, we get that
 \begin{align*}
 \|w\|_{L_t^\infty(L_x^\infty\cap\dot H_x^s)}&\lesssim_{\|u\|_{L_t^\infty X^s}} \|w_0\|_{L_x^\infty\cap\dot H_x^s} + \|f\|_{L^1_t (L^\infty_x \cap \dot H^s_x)}\\
 &\lesssim _{\|u\|_{L_t^\infty X^s}} \|w_0\|_{L_x^\infty\cap\dot H_x^s} + T\|z\|_{L_t^\infty(L_x^\infty\cap\dot H_x^s)}(\|u\|_{L_t^\infty X^s}+\|v\|_{L_t^\infty X^s}).
\end{align*}
 After further shrinking $T$ (depending on $\displaystyle \|(u,v)\|_{L_t^\infty X^s}$, Lemma~\ref{lin-full} implies the desired conclusion.
\end{proof}

 \section{Local well-posedness at low regularity}\label{s:lowreg}
 
 In this section, we prove Theorem \ref{Low regularity local well-posedness}. As we have already noticed at the end of Section~\ref{s:gwp}, this will also imply Theorem \ref{Low regularity global well-posedness}.
 
 \
 
 Let $R>0$ be arbitrary. Given data $u_0$ satisfying $\|u_0\|_{X^s} < R$, we consider the corresponding regularized data
\[
u^h_0 = P_{< h} u_0.
\]
  Since $u_0^h \rightarrow u_0$ in $X^s$, we may assume that $\displaystyle \|u_0^h\|_{X^s}<R$ for all $h$.

We construct a uniform $\dot H_x^1 \cap \dot H_x^{1 + s}$ frequency envelope $\{c_k\}_{k\geq 0}$ for $u_0$ having the following properties:

\begin{enumerate}
     \item[a)]Uniform bounds:     
     \[ \|P_k(u_0^h)_x\|_{H_x^1\cap\dot{H}_x^{1+s}}\lesssim c_k\]
     
     \item[b)]High frequency bounds:     
     \[\|u_0^h\|_{\dot{H}_x^1\cap\dot{H}_x^{2+s}}\lesssim 2^hc_h\]
     
     \item[c)]Difference bounds:     
     \[\|u_0^{h+1}-u_0^h\|_{\dot H^s}\lesssim 2^{-h}c_h\]
     
     \item[d)]Limit as $h\rightarrow\infty$:     
     \[ D_x(u_0^h)\rightarrow D_x(u_0) \in H_x^s\]
     
 \end{enumerate}
 
 \
 
 By Theorem \ref{Global well-posedness} and Lemma \ref{Persistence of regularity}, $u_0^h$ generate global smooth solutions $u^h$. Corollary \ref{Uniform bounds} enables us to pick $T=T(R)>0$ such that the hypotheses of Lemma \ref{Regularity for differences} can be applied to any $C_t^0X^s([0,T]\times\mathbb{R})$-solutions with initial data whose $X^s$-norm is smaller than $R$. Moreover, we also obtain uniform bounds for such solutions, including the family $(u^h)_{h\in\mathbb{Z}}$. We now get that
    \[ 
    \|u^h\|_{C_t^0(\dot{H}_x^1\cap\dot{H}_x^{2+s})}\lesssim 2^{h}c_h,
     \]
 and
\[
\|u^{h+1}-u^h\|_{C_t^0\dot{H}_x^s}\lesssim 2^{-h}c_h
\]
By interpolation, we infer that
\[
\|u^{h+1}-u^h\|_{C_t^0\dot{H}_x^{1+s}}\lesssim c_h.
\]
Thus, for $h\geq 0$,
\[
\|u^{h+1}-u^h\|_{C_t^0(\dot{H}_x^1\cap\dot{H}_x^{1+s})}\lesssim c_h
\]
As in \cite{LWN}, we get that
\[
\|P_k u^h \|_{C_t^0(\dot H_x^1 \cap \dot H_x^{1 + s})} \lesssim c_k.
\]
and that
\[
\|u^{h+k}-u^h\|_{C_t^0(\dot H_x^1 \cap \dot H_x^{1 + s})}\lesssim c_{h\leq\cdot<h+k}=\left(\sum_{\substack{n=h}}^{h+k-1}c_n^2\right)^{\frac{1}{2}}
\]
for every $k\geq 1$. Thus, $u^h$ converges to an element $u$ belonging to $C_t^0(\dot H_x^1 \cap \dot H_x^{1 + s})([0,T]\times\mathbb{R})$.  Moreover, we also obtain
\begin{equation}\label{convergence estimate}
\begin{aligned}
\|u^h - u\|_{C_t^0(\dot H_x^1 \cap \dot H_x^{1 + s})} &\lesssim c_{\geq h}=\left(\sum_{\substack{n=h}}^{\infty} c_n^2\right)^{\frac{1}{2}}.
\end{aligned}
\end{equation}

For pointwise convergence, we use Sobolev embedding for the high frequencies,
\[
\|(u^{h+k})_{>0}-(u^h)_{>0}\|_{C_t^0L_x^\infty}\lesssim\|u^{h+k}-u^h\|_{C_t^0\dot H_x^1}.
\]
and the estimate \eqref{ptwise-diff} for the low frequencies:
\[
\|(u^{h+k})_{\leq 0}-(u^h)_{\leq 0}\|_{C_t^0L_x^\infty}\lesssim \|u_0^{h+k}-u_0^h\|_{L_x^\infty}+TR\|u^{h+k}-u^h\|_{C_t^0\dot H_x^1}.
\]
We conclude that $u^h \rightarrow u\in C_t^0X^s([0,T]\times\mathbb{R})$.

\

Lemma \ref{Regularity for differences} also implies uniqueness for \eqref{HS}. For continuity with respect to the initial data, consider a sequence
 \[
 u_{0j}\rightarrow u_0 \in X^s
 \]
 and an associated sequence of $\displaystyle\dot H_x^1\cap\dot H_x^{1+s}$-frequency envelopes $\{c^j_k\}_{k\geq 0}$, each satisfying the analogous properties enumerated above for $c_k$, and further such that $c^j \rightarrow c$ in $l^2(\mathbb{Z})$.
 
  We may assume that $\|u_{0j}\|_{X^s}<R$ for every $j\geq 0$. As before, we get uniform bounds for $(u_j^h)_{(j,h)\in\mathbb{N}\times\mathbb{Z}}$, and we can interpolate to conclude 
\begin{align*}
\|u_j^{h+1}-u_j^h\|_{C_t^0(\dot{H}_x^1\cap\dot{H}_x^{1+s})}&\lesssim c^j_h
\end{align*}
and
\begin{align*}
\|P_k u_j^h \|_{C_t^0(\dot H_x^1 \cap \dot H_x^{1 + s})}&\lesssim c^j_k, \\
\|u_j^{h+k}-u_j^h\|_{C_t^0(\dot H_x^1 \cap \dot H_x^{1 + s})}&\lesssim c^j_{h\leq\cdot<h+k}=\left(\sum_{\substack{n=h}}^{h+k-1}(c^j_n)^2\right)^{\frac{1}{2}}, \\
\|u_j^h - u_j\|_{C_t^0(\dot H_x^1 \cap \dot H_x^{1 + s})} &\lesssim c^j_{\geq h}=\left(\sum_{\substack{n=h}}^{\infty} (c^j_n)^2\right)^{\frac{1}{2}}.
\end{align*}
 
Using the triangle inequality, we write
\begin{align*}
\|u_j - u\|_{C_t^0(\dot H_x^1 \cap \dot H_x^{1 + s})} &\lesssim \|u^h - u\|_{C_t^0(\dot H_x^1 \cap \dot H_x^{1 + s})}+\|u_j^h - u_j\|_{C_t^0(\dot H_x^1 \cap \dot H_x^{1 + s})}+\|u_j^h - u^h\|_{C_t^0(\dot H_x^1 \cap \dot H_x^{1 + s})}\\
&\lesssim c_{\geq h}+c^j_{\geq h}+\|u_j^h - u^h\|_{C_t^0(\dot H_x^1 \cap \dot H_x^{1 + s})}
\end{align*}
For every fixed $h$, Theorem \ref{Global well-posedness} tells us that $u_j^h \rightarrow u^h$ in $X$. This implies that $u_j \rightarrow u$ in $C_t^0(\dot H_x^1 \cap \dot H_x^{1 + s})([0,T]\times\mathbb{R})$.
For pointwise estimates, by applying Sobolev embeddings and using Bernstein's inequalities, we get that 
\begin{align*}
\|(u_j)_{>0}-u_{>0}\|_{C_t^0 L_x^\infty}\lesssim \|u_j-u\|_{C_t^0(\dot H_x^1\cap \dot H_x^{1+s})}.
\end{align*} 
Besides this, \eqref{ptwise-diff} implies that 
\begin{align*}
    \|(u_j)_{\leq 0}-u_{\leq 0}\|_{C_t^0L_x^\infty}\lesssim\|(u_j)_{\leq 0}-u_{\leq 0}\|_{C_t^0L_x^\infty}+CT\|u_j-u\|_{C_t^0\dot H_x^1}.
\end{align*}
Therefore, $u_j\rightarrow u$ in $C_t^0 X^s([0,T]\times\mathbb{R})$. This finishes the proof.

\bibliographystyle{plain}
\bibliography{Bibliography}

\end{document}